\documentclass[a4paper,12pt,reqno]{amsart}
\pdfoutput=1
\usepackage{graphics}
\usepackage{graphicx}
\usepackage{color}
\usepackage[colorlinks=true]{hyperref}
\usepackage[a4paper,margin=1.9cm,top=2.5cm,bottom=2.5cm,centering,vcentering]{geometry}

\newcommand{\sss}{\scriptscriptstyle}
\allowdisplaybreaks[4]

\newtheorem{theorem}{Theorem}
\newtheorem{corollary}[theorem]{Corollary}
\newtheorem{proposition}[theorem]{Proposition}

\newcommand{\x}{\bar{x}}
\newcommand{\half}{{1/2}}
\newcommand{\oo}{\mathit{so}^{\mathrm{odd}}}
\renewcommand{\oe}{o^{\mathrm{even}}}
\renewcommand{\sp}{\mathit{sp}}

\author[A.~Ayyer]{Arvind Ayyer}
\address{A.~Ayyer, Department of Mathematics, Indian Institute of Science, Bangalore 560 012, India}
\email{arvind@iisc.ac.in}
\author[R.~E.~Behrend]{Roger E.~Behrend}
\address{R.~E.~Behrend, School of Mathematics, Cardiff University, Cardiff, CF24 4AG, UK,
and Fakult\"{a}t f\"{u}r Mathematik, Universit\"{a}t Wien, Oskar-Morgenstern-Platz 1, 1090 Wien, Austria}
\email{behrendr@cardiff.ac.uk}

\title[Factorization theorems for classical group characters]
{Factorization theorems for classical group characters, with
applications to alternating sign matrices\\ 
and plane partitions}
\begin{document}
\begin{abstract}
We show that, for a certain class of partitions
and an even number of variables of which half are reciprocals 
of the other half, Schur polynomials can be factorized into products of odd and even orthogonal characters. 
We also obtain related factorizations involving sums of two Schur polynomials, and certain odd-sized sets of variables.
Our results generalize the factorization identities proved by Ciucu and Krattenthaler 
(Advances in combinatorial mathematics, 39--59, 2009) for partitions of rectangular shape.  
We observe that if, in some of the results, the partitions are taken to have
rectangular or double-staircase shapes and all of the variables are set to~$1$, then factorization identities 
for numbers of certain plane partitions, alternating sign matrices and related combinatorial objects are obtained.
\end{abstract}

\maketitle

\section{Introduction}
Some of the most challenging problems in enumerative combinatorics in recent decades 
have involved plane partitions, alternating sign matrices, and related objects. Schur polynomials, 
i.e., characters of the general linear group, as well as characters of other classical groups, have played a role
in solving many of these problems.  While the appearance of such characters in the enumeration of plane partitions is expected, 
a combinatorial understanding of their occurrence in the enumeration of alternating sign matrices is currently lacking.

The main results of this paper are Theorems~\ref{thm1}--\ref{thm3}, 
which state that, for certain partitions and variables, Schur polynomials can be factorized into products of 
characters of orthogonal groups.  More specifically, in Theorem~\ref{thm1}, we find that,
for any partition $(\lambda_1,\ldots,\lambda_n)$ and integer $k\ge\lambda_1$, or half-partition $(\lambda_1,\ldots,\lambda_n)$ and 
half-integer $k\ge\lambda_1$, the Schur polynomial 
$s_{(k+\lambda_1,\ldots,k+\lambda_n,k-\lambda_n,\ldots,k-\lambda_1)}(x_1,\ldots,x_n,$ $x^{-1}_1,\ldots,x^{-1}_n)$ decomposes, up to a simple
prefactor, into a product of a character of a special odd orthogonal group and a character of an even orthogonal group.  In Theorem~\ref{thm2},
we obtain a similar factorization of the sum of two Schur polynomials, and in Theorem~\ref{thm3}, we obtain 
a factorization of a Schur polynomial with arguments $x_1,\ldots,x_n,x^{-1}_1,\ldots,x^{-1}_n,1$. 
We prove all of these cases by using standard determinant expressions (which arise from the Weyl character formula) 
for the characters of the classical groups, and applying elementary determinant operations. 
Currently, we do not have a representation-theoretic explanation of these factorizations, although it would clearly be interesting to find one. 
On the other hand, combinatorial proofs of such factorizations,
and of generalizations involving skew Schur polynomials, will be provided in a forthcoming paper~\cite{AyyFis18}.

The connection between these results and plane partitions, alternating sign matrices and related combinatorial objects is that factorization identities 
for numbers of such objects can be obtained by setting all of the variables in the theorems to~$1$, and specializing to partitions with
certain rectangular or double-staircase shapes.
Indeed, the work reported in this paper was primarily motivated by previously-known factorization identities 
for Schur polynomials indexed by 
partitions with such shapes.  In particular, identities for the rectangular case
$s_{(\underbrace{\scriptstyle m,\ldots,m}_n,\underbrace{\scriptstyle 0,\ldots,0}_n)}(x_1,\ldots,x_n,x^{-1}_1,\ldots,x^{-1}_n)$
were noted without proof by Okada~\cite[Lem.~5.2, 1st Eq.]{Oka98}, and proved by Ciucu and Krattenthaler~\cite[Thms.~3.1 \& 3.2]{CiuKra09}, 
while identities for the double-staircase cases	
$s_{(n,n-1,n-1,\ldots,1,1,0)}(x_1,\ldots,x_n,x^{-1}_1,\ldots,x^{-1}_n)$, \ 
$s_{(n-1,n-1,\ldots,1,1,0,0)}(x_1,\ldots,x_n,x^{-1}_1,\ldots,\allowbreak x^{-1}_n)$ and
$s_{(n,n-1,n-1,\ldots,1,1,0,0)}(x_1,\ldots,x_n,x^{-1}_1,\ldots,x^{-1}_n,1)$ were noted without proof by
Ayyer, Beh\-rend and Fischer~\cite[Remarks~5.4 \& 6.4]{AyyBehFis16}, and Behrend, Fischer and Konvalinka~\cite[Eq.~(63)]{BehFisKon17}.
The results of this paper generalize all of these cases, and provide new, shorter proofs 
for the rectangular cases.

The structure of the rest of this paper is as follows. In Section~\ref{characters}, 
we introduce the classical group characters, and identify some of their elementary properties.
In Section~\ref{results}, we state the main Schur polynomial factorization theorems, and
in Section~\ref{proofs} we provide their proofs.
In Section~\ref{applications}, we apply the main theorems to 
partitions with rectangular or double-staircase shapes,
and obtain factorization identities for numbers of plane partitions, 
alternating sign matrices and related objects.

\section{Classical group characters}\label{characters}
We first review some standard terminology.
A \emph{half-integer} is an odd integer divided by~2.
A \emph{partition}, or respectively \emph{half-partition}, is a tuple $(\lambda_1,\ldots,\lambda_n)$ whose 
entries are all nonnegative integers,
or respectively all positive half-integers, in weakly decreasing order, $\lambda_1\ge\ldots\ge\lambda_n$. 
In the literature, zero entries of a partition are often omitted, but in this paper 
it will be convenient for these to be shown explicitly, since they often play an important role.  
The \emph{shape} of a partition refers to the shape of its associated Young diagram. For further information regarding 
partitions and Young diagrams, see, for example, Macdonald~\cite[Ch.~I.1]{Mac95} or Stanley~\cite[Ch.~7]{Sta99},~\cite[Ch.~1]{Sta11}.

For an indeterminate $x$, the notation 
\begin{equation*}\x=x^{-1}\end{equation*}
will be used throughout the rest of this paper.  

We now introduce the classical group characters. 
For indeterminates $x_1,\ldots,x_n$, and a partition $\lambda=(\lambda_1,\ldots,\lambda_n)$, the \emph{Schur polynomial} 
or \emph{general linear character} is given by
\begin{equation}\label{Schurdef}s_\lambda(x_1,\ldots,x_n)=\frac{\displaystyle\det_{1\le i,j\le n}\Bigl(x_i^{\lambda_j+n-j}\Bigr)}
{\prod_{1\le i<j\le n}(x_i-x_j)},\end{equation}
and the \emph{symplectic character} is given by
\begin{equation}\label{spdef}\sp_\lambda(x_1,\ldots,x_n)=
\frac{\displaystyle\det_{1\le i,j\le n}\Bigl(x_i^{\lambda_j+n-j+1}-\x_i^{\lambda_j+n-j+1}\Bigr)}
{\prod_{i=1}^n(x_i-\x_i)\,\prod_{1\le i<j\le n}(x_i+\x_i-x_j-\x_j)}.\end{equation}

For indeterminates $x_1,\ldots,x_n$, and a partition or half-partition $\lambda=(\lambda_1,\ldots,\lambda_n)$, the \emph{odd orthogonal character} is given by
\begin{equation}\label{oodef}\oo_\lambda(x_1,\ldots,x_n)=
\frac{\displaystyle\det_{1\le i,j\le n}\Bigl(x_i^{\lambda_j+n-j+\half}-\x_i^{\lambda_j+n-j+\half}\Bigr)}
{\prod_{i=1}^n\bigl(x_i^\half-\x_i^\half\bigr)\,\prod_{1\le i<j\le n}(x_i+\x_i-x_j-\x_j)},\end{equation}
and the \emph{even orthogonal character} is given by
\begin{equation}\label{oedef}\oe_\lambda(x_1,\ldots,x_n)=
\frac{\displaystyle\det_{1\le i,j\le n}\Bigl(x_i^{\lambda_j+n-j}+\x_i^{\lambda_j+n-j}\Bigr)}
{(1+\delta_{\lambda_n,0})\prod_{1\le i<j\le n}(x_i+\x_i-x_j-\x_j)},\end{equation}
where $\delta$ is the Kronecker delta.

For the exact connection between the functions~\eqref{Schurdef}--\eqref{oedef}
and characters of irreducible representations of the general linear group $\mathrm{GL}_n(\mathbb{C})$, symplectic group $\mathrm{Sp}_{2n}(\mathbb{C})$,
special odd orthogonal group $\mathrm{S0}_{2n+1}(\mathbb{C})$ and even orthogonal group
$\mathrm{0}_{2n}(\mathbb{C})$ (and the spin covering groups for the orthogonal case), see, for example, Fulton and Harris~\cite[Ch.~24]{FulHar91} 
and Proctor~\cite[Appendix~2]{Pro94}.

Some elementary properties of these functions, which can be verified using the expressions \eqref{Schurdef}--\eqref{oedef}, are as follows,
where $\lambda$ denotes $(\lambda_1,\ldots,\lambda_n)$.
\begin{list}{$\bullet$}{\setlength{\topsep}{0.8mm}\setlength{\labelwidth}{2mm}\setlength{\leftmargin}{6mm}}
\item For a partition $\lambda$, $s_{\lambda}(x_1,\ldots,x_n)$ is a symmetric polynomial in
$x_1$, \ldots,~$x_n$, and $\sp_\lambda(x_1,\ldots,x_n)$, $\oo_\lambda(x_1,\ldots,x_n)$ and
$\oe_\lambda(x_1,\ldots,x_n)$ are symmetric Laurent polynomials in $x_1$, \ldots, $x_n$.
\item For a half-partition $\lambda$, $\oo_\lambda(x_1^2,\ldots,x_n^2)$ and
$\oe_\lambda(x_1^2,\ldots,x_n^2)$ are symmetric Laurent polynomials in $x_1$, \ldots, $x_n$.
\item For a partition $\lambda$ and integer $k\ge-\lambda_n$,
\begin{equation}\label{Schurtrans}(x_1\ldots x_n)^k\,s_\lambda(x_1,\ldots,x_n)=
s_{(k+\lambda_1,\ldots,k+\lambda_n)}(x_1,\ldots,x_n).\end{equation}
\item For a partition $\lambda$ and integer $k\ge\lambda_1$,
\begin{equation}\label{Schurrecip}
(x_1\ldots x_n)^k\,s_\lambda(\x_1,\ldots,\x_n)=s_{(k-\lambda_n,\ldots,k-\lambda_1)}(x_1,\ldots,x_n).\end{equation}
\item For a partition $\lambda$ and $1\le i\le n$,
$\sp_\lambda(x_1,\ldots,x_n)$ is invariant under replacement of $x_i$ with $\x_i$.
\item For a partition or half-partition $\lambda$ and $1\le i\le n$,
$\oo_\lambda(x_1,\ldots,x_n)$
and $\oe_\lambda(x_1,\ldots,x_n)$ are invariant under replacement of $x_i$ with $\x_i$.
\item For a partition $\lambda$, the odd and even orthogonal
characters indexed by the half-partition $(\lambda_1+\half,\ldots,\lambda_n+\half)$ can
be expressed in terms of characters indexed by $\lambda$ as
\begin{align}\label{ooshifted}\quad\oo_{(\lambda_1+\half,\ldots,\lambda_n+\half)}(x_1,\ldots,x_n)&=
\prod_{i=1}^n\bigl(x_i^\half+\x_i^\half\bigr)\:\sp_\lambda(x_1,\ldots,x_n)
\intertext{and}
\label{oeshifted}\oe_{(\lambda_1+\half,\ldots,\lambda_n+\half)}(x_1,\ldots,x_n)&=
(-1)^{\sum_{i=1}^n\lambda_i}\,\prod_{i=1}^n\bigl(x_i^\half+\x_i^\half\bigr)\:\oo_\lambda(-x_1,\ldots,-x_n).\end{align}
Note that ambiguities arise in $(-x_i)^k$ if $k$ is a half-integer,
and hence the consideration of such terms will be avoided in this paper. This can be done for the odd orthogonal character 
in~\eqref{oeshifted} by observing that (since $\lambda$ is a partition)
$\oo_\lambda(x_1,\ldots,x_n)$ is a Laurent polynomial in $x_1,\ldots,x_n$, which can then be evaluated at $-x_1,\ldots,-x_n$.
\end{list}

\section{Main results}\label{results}
In this section, we state the main results of the paper. In Subsection~\ref{mainstate}, we
provide the primary statements of these results, while, in 
Subsection~\ref{altstate}, we provide alternative statements of the results.
Proofs will be given in Section~\ref{proofs}.

\subsection{Primary statements of the results}\label{mainstate}
The main results of this paper are as follows.  (Some general remarks regarding these results will be given
at the end of this subsection.)
\begin{theorem}\label{thm1} For any partition $(\lambda_1,\ldots,\lambda_n)$ and integer $k\ge\lambda_1$,
or half-partition $(\lambda_1,\ldots,\lambda_n)$ and half-integer $k\ge\lambda_1$,
\begin{multline}\label{fact1}\prod_{i=1}^n\bigl(x_i^\half+\x_i^\half\bigr)\,
s_{(k+\lambda_1,\ldots,k+\lambda_n,k-\lambda_n,\ldots,k-\lambda_1)}(x_1,\ldots,x_n,\x_1,\ldots,\x_n)\\*[-2mm]
=\oo_{(\lambda_1,\ldots,\lambda_n)}(x_1,\ldots,x_n)\;\oe_{(\lambda_1+\half,\ldots,\lambda_n+\half)}(x_1,\ldots,x_n).\end{multline}
\end{theorem}

For the case in which $(\lambda_1,\ldots,\lambda_n)$ is a partition 
and $k$ is an integer, the shape of the partition on the LHS of~\eqref{fact1} is illustrated in Figure~\ref{thm1fig}.

\begin{figure}[htbp!]
\begin{center}
\includegraphics[scale=0.5]{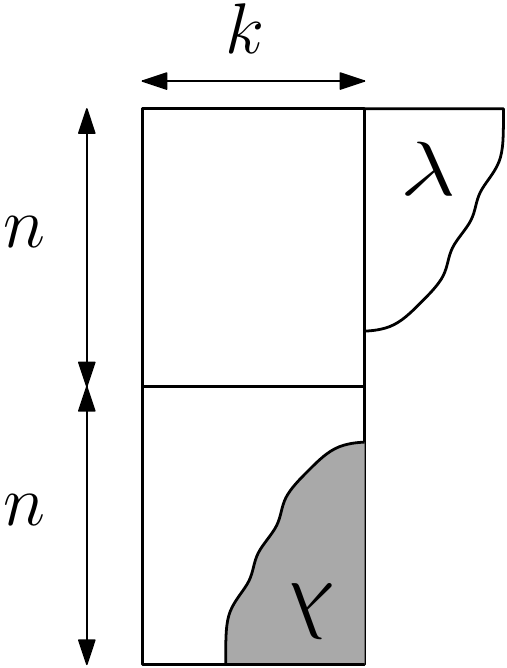}
\caption{The shape of the partition on the LHS of~\eqref{fact1},
where $\lambda=(\lambda_1,\ldots,\lambda_n)$ is a partition and $k\ge\lambda_1$ is an integer.
The inverted $\lambda$ denotes $(\lambda_n,\ldots,\lambda_1)$, and 
the shading indicates that the corresponding region should be removed.}
\label{thm1fig}
\end{center}
\end{figure}

\begin{theorem}\label{thm2} For any partition $(\lambda_0,\ldots,\lambda_n)$ and integers $k_1,k_2\ge\lambda_0$,
or half-partition $(\lambda_0,\ldots,\allowbreak\lambda_n)$ and half-integers $k_1,k_2\ge\lambda_0$,
\begin{multline}\label{fact2}\prod_{i=1}^n\bigl(x_i^\half+\x_i^\half\bigr)\,
\bigl(s_{(k_1+\lambda_1,\ldots,k_1+\lambda_n,k_1-\lambda_{n-1},\ldots,k_1-\lambda_0)}(x_1,\ldots,x_n,\x_1,\ldots,\x_n)\\*[-3mm]
\qquad\qquad\qquad{}+
s_{(k_2+\lambda_1,\ldots,k_2+\lambda_{n-1},k_2-\lambda_n,\ldots,k_2-\lambda_0)}(x_1,\ldots,x_n,\x_1,\ldots,\x_n)\bigr)\\*[1mm]
=(1+\delta_{\lambda_n,0})\,
\oo_{(\lambda_0+\half,\ldots,\lambda_{n-1}+\half)}(x_1,\ldots,x_n)\;\oe_{(\lambda_1,\ldots,\lambda_n)}(x_1,\ldots,x_n).\end{multline}
\end{theorem}

For the case in which $(\lambda_0,\ldots,\lambda_n)$ is a partition
and $k_1$, $k_2$ are integers, the shapes of the partitions on the LHS of~\eqref{fact2} are illustrated in Figure~\ref{thm2fig}.

\begin{figure}[htbp!]
\begin{center}
\includegraphics[scale=0.5]{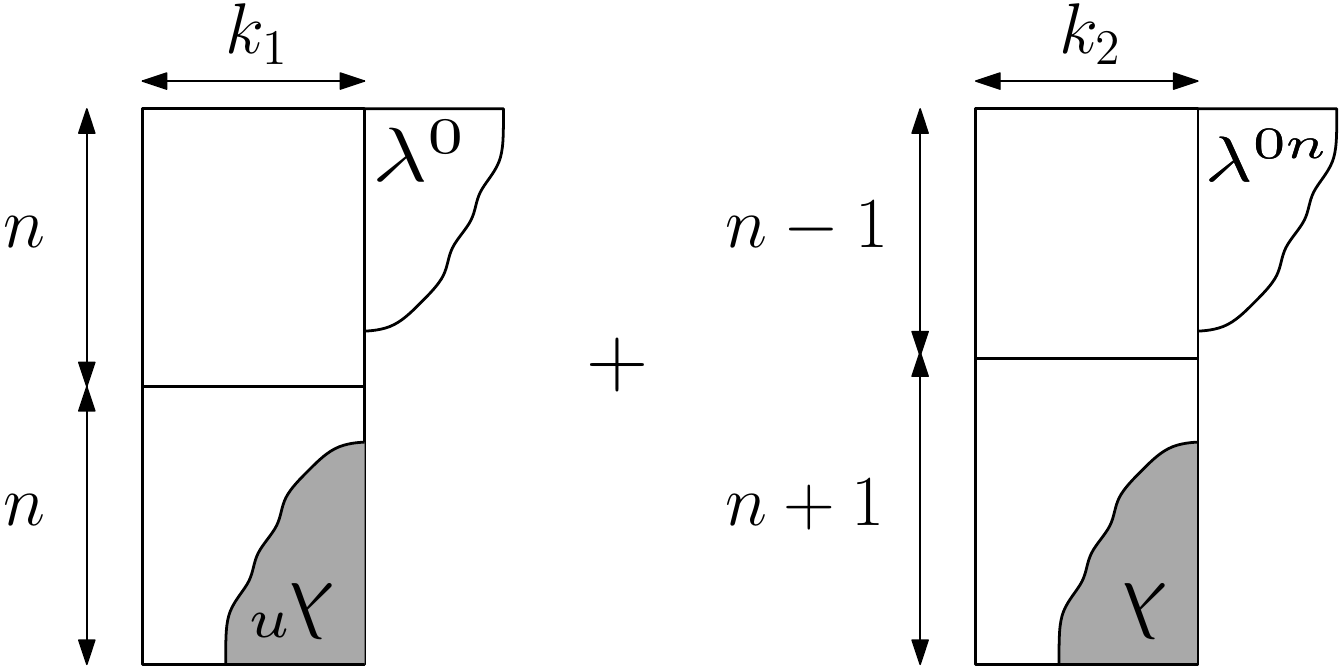}
\caption{The shapes of the partitions on the LHS of~\eqref{fact2}, where 
$\lambda=(\lambda_0,\ldots,\lambda_n)$ is a partition,
$\lambda^0$, $\lambda^n$ and $\lambda^{0n}$ denote $(\lambda_1,\ldots,\lambda_n)$,
$(\lambda_0,\ldots,\lambda_{n-1})$ and $(\lambda_1,\ldots,\lambda_{n-1})$, respectively,
and $k_1,k_2\ge\lambda_0$ are integers.
Also, the inverted $\lambda$ and $\lambda^n$ denote $(\lambda_n,\ldots,\lambda_0)$
and $(\lambda_{n-1},\ldots,\lambda_0)$, respectively, and 
shading indicates that the corresponding region should be removed.}
\label{thm2fig}
\end{center}
\end{figure}

\begin{theorem}\label{thm3} For any partition $(\lambda_0,\ldots,\lambda_n)$ and integer $k\ge\lambda_0$,
or half-partition $(\lambda_0,\ldots,\lambda_n)$ and half-integer $k\ge\lambda_0$,
\begin{align}\notag&2\,\prod_{i=1}^n\bigl(x_i^\half+\x_i^\half\bigr)\,
s_{(k+\lambda_1,\ldots,k+\lambda_n,k-\lambda_n,\ldots,k-\lambda_0)}(x_1,\ldots,x_n,\x_1,\ldots,\x_n,1)\\*[-2mm]
\label{fact3}&\hspace{38mm}=
\oo_{(\lambda_1,\ldots,\lambda_n)}(x_1,\ldots,x_n)\;\oe_{(\lambda_0+\half,\ldots,\lambda_n+\half)}(x_1,\ldots,x_n,1)
\intertext{and}
\notag&2\,\prod_{i=1}^n\bigl(x_i^\half+\x_i^\half\bigr)\,
s_{(k+\lambda_0,\ldots,k+\lambda_{n-1},k\pm\lambda_n,k-\lambda_{n-1},\ldots,k-\lambda_0)}(x_1,\ldots,x_n,\x_1,\ldots,\x_n,1)\\*[-2mm]
\label{fact3'}&\hspace{38mm}=(1+\delta_{\lambda_n,0})\,
\oo_{(\lambda_0+\half,\ldots,\lambda_{n-1}+\half)}(x_1,\ldots,x_n)\;\oe_{(\lambda_0,\ldots,\lambda_n)}(x_1,\ldots,x_n,1).\end{align}
\end{theorem}

Note that the $\pm$ on the LHS of~\eqref{fact3'} indicates that either $+$ or~$-$ can be used to give a valid equation.

For the case in which $(\lambda_0,\ldots,\lambda_n)$ is a partition and $k$ is an integer, the shapes of the partitions on 
the LHSs of~\eqref{fact3} and \eqref{fact3'} are illustrated in Figure~\ref{thm3fig}.

\begin{figure}[htbp!]
\begin{center}
\begin{tabular}{c@{\;\quad}|@{\quad}c@{\;\quad}|@{\quad}c}
\includegraphics[scale=0.5]{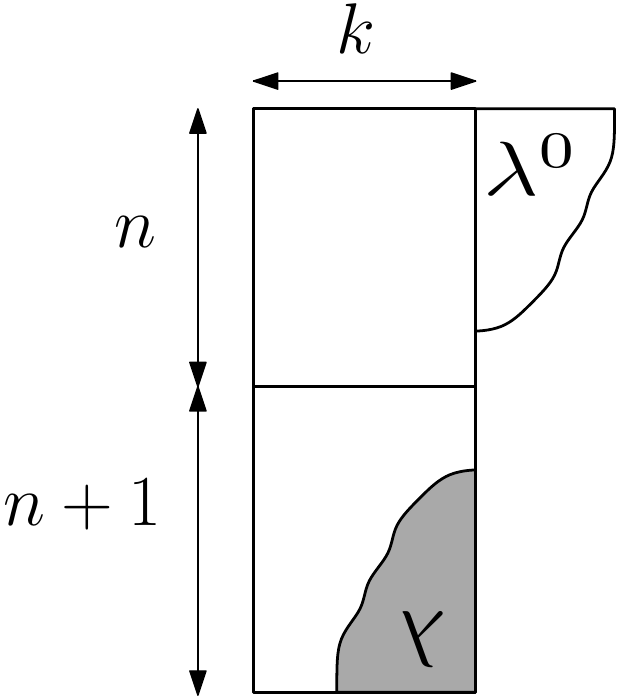}&\includegraphics[scale=0.5]{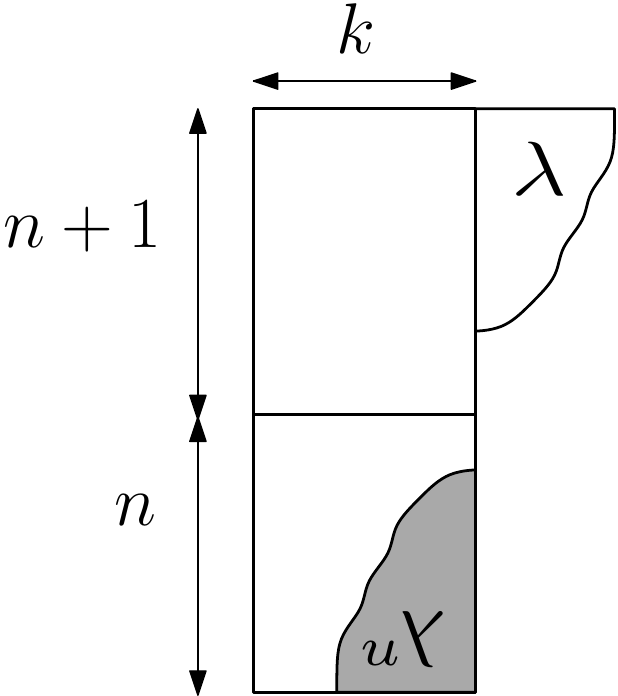}&\includegraphics[scale=0.5]{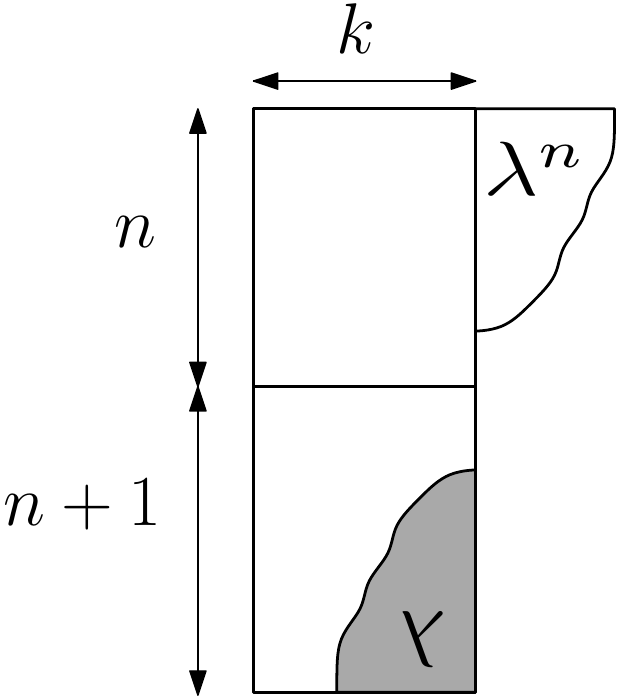}\end{tabular}
\caption{The shapes of the partitions on the LHSs of~\eqref{fact3} and \eqref{fact3'}, where 
$\lambda=(\lambda_0,\ldots,\lambda_n)$ is a partition,
$\lambda^0$ denotes $(\lambda_1,\ldots,\lambda_n)$,
$\lambda^n$ denotes $(\lambda_0,\ldots,\lambda_{n-1})$,
and $k\ge\lambda_0$ is an integer.
Also, the inverted $\lambda$ and $\lambda^n$ denote $(\lambda_n,\ldots,\lambda_0)$
and $(\lambda_{n-1},\ldots,\lambda_0)$, respectively, and 
shading indicates that the corresponding region should be removed.
The three diagrams correspond to~\eqref{fact3}, 
the $+$ case of~\eqref{fact3'} and the~$-$ case of~\eqref{fact3'}, respectively.}
\label{thm3fig}\end{center}\end{figure}

In certain cases, the even orthogonal characters in Theorem~\ref{thm3} can be expressed, up to simple prefactors, as odd orthogonal characters, 
as stated in the following results. 
\begin{proposition}\label{prop1}
If  
\begin{equation}\label{lam4}(\lambda_0,\ldots,\lambda_n)=(nb+a,\ldots,2b+a,b+a,a),\end{equation}
for a nonnegative integer or positive half-integer $a$, and a nonnegative integer~$b$, then
\begin{equation}\label{oeodd1}(1+\delta_{\lambda_n,0})\,
\oe_{(\lambda_0,\ldots,\lambda_n)}(x_1,\ldots,x_n,1)=2\prod_{i=1}^n\Biggl(\,\sum_{j=-b/2}^{b/2}x_i^j\!\Biggr)\,
\oo_{(\lambda_0-b/2,\ldots,\lambda_{n-1}-b/2)}(x_1,\ldots,x_n).\end{equation}
\end{proposition}

Note that if $b$ is odd, then $j$ in the sum on the RHS of~\eqref{oeodd1} ranges over all half-integers from~$-b/2$ to~$b/2$.

\begin{proposition}\label{prop2}
If 
\begin{equation}\label{lam5}\textstyle(\lambda_0,\ldots,\lambda_n)=\bigl(\lfloor\frac{n+1}{2}\rfloor b+(-1)^na,\ldots,2b+a,2b-a,b+a,b-a,a\bigr),\end{equation}
for a nonnegative integer or positive half-integer $a$, and an integer $b\ge2a$, then
\begin{equation}\label{oeodd2}\quad\oe_{(\lambda_0+\half,\ldots,\lambda_n+\half)}(x_1,\ldots,x_n,1)
=2\prod_{i=1}^n\Biggl(\,\sum_{j=-(b+1)/2}^{(b+1)/2}x_i^j\!\Biggr)\,
\oo_{(\lambda_0-b/2,\ldots,\lambda_{n-1}-b/2)}(x_1,\ldots,x_n).\end{equation}
\end{proposition}

Note that if $b$ is even, then $j$ in the sum on the RHS of~\eqref{oeodd2} ranges over all half-integers from~$-(b+1)/2$ to~$(b+1)/2$.

By slightly modifying the proofs of Theorems~\ref{thm1}--\ref{thm3},
it is possible to obtain certain closely related results. An example is the following variation of Theorems~\ref{thm1} and~\ref{thm3},
for which we omit the proof.
For any partition $(\lambda_1,\ldots,\lambda_n)$ and integer $k\ge\lambda_1$,
or half-partition $(\lambda_1,\ldots,\lambda_n)$ and half-integer $k\ge\lambda_1$,
\begin{multline}\label{fact4}2\,\prod_{i=1}^{n-1}\bigl(x_i^\half+\x_i^\half\bigr)\,
s_{(k+\lambda_1,\ldots,k+\lambda_n,k-\lambda_n,\ldots,k-\lambda_1)}(x_1,\ldots,x_n,\x_1,\ldots,\x_{n-1},1)\\*[-2mm]
=x_n^k\,\oo_{(\lambda_1,\ldots,\lambda_n)}(x_1,\ldots,x_n)\;\oe_{(\lambda_1+\half,\ldots,\lambda_n+\half)}(x_1,\ldots,x_{n-1},1).\end{multline}

Some remarks regarding the previous results are as follows.
\begin{list}{$\bullet$}{\setlength{\topsep}{0.8mm}\setlength{\labelwidth}{2mm}\setlength{\leftmargin}{6mm}}
\item It can easily be checked that each of the Schur polynomials is indexed by
a valid partition, and each of the orthogonal characters is indexed by a valid partition or half-partition.
\item For a partition $(\lambda_1,\ldots,\lambda_a)$ with $\lambda_1\le b$, the \emph{complement} of $(\lambda_1,\ldots,\lambda_a)$ in
an $a\times b$ rectangle is defined to be the partition $(b-\lambda_a,\ldots,b-\lambda_1)$.  It can easily be checked that, for fixed~$n$ and~$k$, 
the partitions which can occur on the LHS of~\eqref{fact1} are precisely those which are self-complementary
in a $(2n)\times(2k)$ rectangle (where $2n$ is necessarily even, since $n$ is an integer, but $2k$ can be even or odd, since $k$ 
can be an integer or half-integer).
Similarly, for fixed~$n$ and~$k$, and $\lambda_n$ set to~$0$ (which implies that 
$(\lambda_0,\ldots,\lambda_n)$ is a partition and $k$ is an integer), the partitions which can occur on the LHS of~\eqref{fact3'} are 
precisely those which are self-complementary in a $(2n+1)\times(2k)$ rectangle
(where $2n+1$ and $2k$ are necessarily odd and even, respectively). It follows that the LHSs 
of~\eqref{fact1} and~\eqref{fact3'} include all self-complementary partitions, since a partition cannot be 
self-complementary in an $a\times b$ rectangle if $a$ and $b$ are both odd.
\item Partitions with the same forms as those on the LHSs of~\eqref{fact1}, ~\eqref{fact3'} and~\eqref{fact4}
have appeared previously in results of Stanley~\cite[Lem.~3.3]{Sta86a}, \cite[Exer.~7.106a]{Sta99}.
\item The RHSs of~\eqref{fact1},~\eqref{fact3} and~\eqref{fact3'} are explicitly independent of $k$, and
this independence for each LHS follows from~\eqref{Schurtrans}.
Similarly, both sides of~\eqref{fact2} are independent of~$k_1$ and~$k_2$.
\item Due to~\eqref{Schurtrans} and~\eqref{Schurrecip} (and the fact that a Schur polynomial is a symmetric function), 
Theorem~\ref{thm2} remains valid if the condition
$k_1\ge\lambda_0$ is replaced by $k_1\ge\lambda_1$ and the first Schur polynomial in~\eqref{fact2} is
replaced by $s_{(k_1+\lambda_0,\ldots,k_1+\lambda_{n-1},k_1-\lambda_n,\ldots,k_1-\lambda_1)}(x_1,\ldots,x_n,\x_1,\ldots,\x_n)$,
or the condition $k_2\ge\lambda_0$ is replaced by $k_2\ge\lambda_1$ and the second Schur polynomial in~\eqref{fact2}
is replaced by $s_{(k_2+\lambda_0,\ldots,k_2+\lambda_n,k_2-\lambda_{n-1},\ldots,k_2-\lambda_1)}(x_1,\ldots,x_n,\x_1,\ldots,\x_n)$.
Hence, the apparent asymmetry in the structure of each of the partitions on the LHS of~\eqref{fact2} can be reversed.
\item Similarly,~\eqref{fact3} remains valid if the condition $k\ge\lambda_0$
is replaced by $k\ge\lambda_1$, and the Schur polynomial is replaced by
$s_{(k+\lambda_0,\ldots,k+\lambda_n,k-\lambda_n,\ldots,k-\lambda_1)}(x_1,\ldots,x_n,\x_1,\ldots,\x_n,1)$.
\item The equality between the $+$ and $-$ cases of the Schur polynomial in~\eqref{fact3'} follows from~\eqref{Schurrecip}.
\end{list}

\subsection{Alternative statements of the results}\label{altstate}
In each of~\eqref{fact1}--\eqref{fact3'} and~\eqref{fact4}, and in certain cases of~\eqref{oeodd1} and~\eqref{oeodd2}, one of
the two orthogonal characters is indexed by a partition and the other is indexed by a half-partition.  However, 
by using~\eqref{ooshifted} and~\eqref{oeshifted}, each of these results can be expressed in a form in which all of the 
characters are indexed by partitions. This will now be done for Theorems~\ref{thm1} and~\ref{thm2}, and the other cases will then
be discussed briefly. For simplicity, and for later convenience in Section~\ref{applications}, the alternative statements of 
Theorems~\ref{thm1} and~\ref{thm2} will assume that $k=\lambda_1$ and $k_1=k_2=\lambda_0$.

A restatement of Theorem~\ref{thm1} (with $k=\lambda_1$) is as follows. For any partition $(\lambda_1,\ldots,\lambda_n)$,
\begin{align}\notag&s_{(\lambda_1+\lambda_1,\ldots,\lambda_1+\lambda_n,\lambda_1-\lambda_n,\ldots,\lambda_1-\lambda_1)}(x_1,\ldots,x_n,\x_1,\ldots,\x_n)\\*
\label{fact11}&\hspace{50mm}=(-1)^{\sum_{i=1}^n\lambda_i}\,\oo_{(\lambda_1,\ldots,\lambda_n)}(x_1,\ldots,x_n)\;
\oo_{(\lambda_1,\ldots,\lambda_n)}(-x_1,\ldots,-x_n)
\intertext{and}
\notag&s_{(\lambda_1+\lambda_1+1,\ldots,\lambda_1+\lambda_n+1,\lambda_1-\lambda_n,\ldots,\lambda_1-\lambda_1)}(x_1,\ldots,x_n,\x_1,\ldots,\x_n)\\*
\label{fact12}&\hspace{50mm}=\sp_{(\lambda_1,\ldots,\lambda_n)}(x_1,\ldots,x_n)\;\oe_{(\lambda_1+1,\ldots,\lambda_n+1)}(x_1,\ldots,x_n).\end{align}
Note that~\eqref{fact11} corresponds to the case of Theorem~\ref{thm1} in which $(\lambda_1,\ldots,\lambda_n)$ is a partition (and~$k=\lambda_1$) 
in~\eqref{fact1}, and~\eqref{oeshifted} is applied to the even orthogonal character.
On the other hand,~\eqref{fact12} corresponds to the case of Theorem~\ref{thm1}
in which $(\lambda_1,\ldots,\lambda_n)$ is a half-partition (and~$k=\lambda_1$) in~\eqref{fact1}.  
In this case,~\eqref{fact12} is obtained by replacing $(\lambda_1,\ldots,\lambda_n)$ 
by $\bigl(\lambda_1+\frac{1}{2},\ldots,\lambda_n+\frac{1}{2}\bigr)$, and applying~\eqref{ooshifted} to the odd orthogonal character.

A restatement of Theorem~\ref{thm2} (with $k_1=k_2=\lambda_0$) is as follows. For any partition $(\lambda_0,\ldots,\lambda_n)$,
\begin{align}\notag&s_{(\lambda_0+\lambda_1,\ldots,\lambda_0+\lambda_n,\lambda_0-\lambda_{n-1},\ldots,\lambda_0-\lambda_0)}(x_1,\ldots,x_n,\x_1,\ldots,\x_n)\\*
\notag&\hspace{22mm}+
s_{(\lambda_0+\lambda_1,\ldots,\lambda_0+\lambda_{n-1},\lambda_0-\lambda_n,\ldots,\lambda_0-\lambda_0)}(x_1,\ldots,x_n,\x_1,\ldots,\x_n)\\*
\label{fact21}&\hspace{38mm}=(1+\delta_{\lambda_n,0})\,
\sp_{(\lambda_0,\ldots,\lambda_{n-1})}(x_1,\ldots,x_n)\;\oe_{(\lambda_1,\ldots,\lambda_n)}(x_1,\ldots,x_n)
\intertext{and}
\notag&s_{(\lambda_0+\lambda_1+1,\ldots,\lambda_0+\lambda_n+1,\lambda_0-\lambda_{n-1},\ldots,\lambda_0-\lambda_0)}(x_1,\ldots,x_n,\x_1,\ldots,\x_n)\\*
\notag&\hspace{22mm}+
s_{(\lambda_0+\lambda_1+1,\ldots,\lambda_0+\lambda_{n-1}+1,\lambda_0-\lambda_n,\ldots,\lambda_0-\lambda_0)}(x_1,\ldots,x_n,\x_1,\ldots,\x_n)\\*
\label{fact22}&\hspace{38mm}
=(-1)^{\sum_{i=1}^n\lambda_i}\,\oo_{(\lambda_0+1,\ldots,\lambda_{n-1}+1)}(x_1,\ldots,x_n)\;\oo_{(\lambda_1,\ldots,\lambda_n)}(-x_1,\ldots,-x_n).\end{align}
Note that~\eqref{fact21} corresponds to the case of Theorem~\ref{thm2} in which $(\lambda_0,\ldots,\lambda_n)$ is a partition (and~$k_1=k_2=\lambda_0$) 
in~\eqref{fact2}, and~\eqref{ooshifted} is applied to the odd orthogonal character.
On the other hand,~\eqref{fact22} corresponds to the case of Theorem~\ref{thm2}
in which $(\lambda_0,\ldots,\lambda_n)$ is a half-partition (and~$k_1=k_2=\lambda_0$) in~\eqref{fact2}.
In this case,~\eqref{fact22} is obtained by replacing $(\lambda_0,\ldots,\lambda_n)$ 
by $\bigl(\lambda_0+\frac{1}{2},\ldots,\lambda_n+\frac{1}{2}\bigr)$,
and applying~\eqref{oeshifted} to the even orthogonal character.

Theorem~\ref{thm3}, Propositions~\ref{prop1} and~\ref{prop2}, and~\eqref{fact4} can also be restated in forms analogous to 
those given above for Theorems~\ref{thm1} and~\ref{thm2}.  

The case of Theorem~\ref{thm3} in which $(\lambda_0,\ldots,\lambda_n)$ is a partition 
and~$k$ is an integer gives reformulations of~\eqref{fact3} and~\eqref{fact3'} 
in which~\eqref{oeshifted} is applied to the even orthogonal character in~\eqref{fact3},
and~\eqref{ooshifted} is applied to the odd orthogonal character in~\eqref{fact3'}.
The case of Theorem~\ref{thm3} in which $(\lambda_0,\ldots,\lambda_n)$ is a half-partition 
and~$k$ is a half-integer gives reformulations of~\eqref{fact3} and~\eqref{fact3'} 
in which~\eqref{ooshifted} is applied to the odd orthogonal character in~\eqref{fact3},
and~\eqref{oeshifted} is applied to the even orthogonal character in~\eqref{fact3'}.

Propositions~\ref{prop1} and~\ref{prop2} lead to four cases each. For 
example, these are as follows for Proposition~\ref{prop1}.
If $a$ is an integer and $b$ is even, then~\eqref{oeodd1} is left unchanged. 
If $a$ is an integer and $b$ is odd, then~\eqref{ooshifted} is applied to the odd orthogonal character in~\eqref{oeodd1}.
If $a$ is a half-integer and $b$ is odd, then~\eqref{oeshifted} is applied to the even orthogonal character in~\eqref{oeodd1}.
If $a$ is a half-integer and~$b$ is even, then~\eqref{ooshifted} and~\eqref{oeshifted} are applied to the odd and even orthogonal characters,
respectively, in~\eqref{oeodd1}.

\section{Proofs}\label{proofs}
In this section, we provide proofs of Theorems~\ref{thm1}--\ref{thm3} and Propositions~\ref{prop1} and~\ref{prop2}.
Note that, essentially, these proofs rely only on the determinant expressions~\eqref{Schurdef},~\eqref{oodef} 
and~\eqref{oedef} for Schur polynomials and orthogonal characters, and
the application of standard determinant operations.
\begin{proof}[Proof of Theorem~\ref{thm1}]
Using the Schur polynomial expression \eqref{Schurdef}, the Schur polynomial on the LHS of~\eqref{fact1} is
\begin{multline}\label{det1}
s_{(k+\lambda_1,\ldots,k+\lambda_n,k-\lambda_n,\ldots,k-\lambda_1)}(x_1,\ldots,x_n,\x_1,\ldots,\x_n)\\[1mm]
=\frac{\rule[-8mm]{0mm}{0mm}\det\left(\begin{array}{c|c}
\bigl(x_i^{k+\lambda_j+2n-j}\bigr)_{\sss1\le i,j\le n}&\bigl(x_i^{k-\lambda_{n+1-j}+n-j}\bigr)_{\sss1\le i,j\le n}\\[2mm]\hline\rule{0mm}{6mm}
\bigl(\x_i^{k+\lambda_j+2n-j}\bigr)_{\sss1\le i,j\le n}&\bigl(\x_i^{k-\lambda_{n+1-j}+n-j}\bigr)_{\sss1\le i,j\le n}\end{array}\right)}
{\rule{0mm}{4mm}\prod_{i=1}^n(x_i-\x_i)\,\prod_{1\le i<j\le n}(x_i-x_j)(\x_i-\x_j)(x_i-\x_j)(x_j-\x_i)}.\end{multline}

By multiplying row~$i$ in the top blocks of the matrix in~\eqref{det1} by $\x_i^{k+n-\half}$
and row~$i$ in the bottom blocks by $x_i^{k+n-\half}$, for each $i=1,\ldots,n$, and then reversing the order of the columns
in the right blocks, it follows that the numerator of the RHS of~\eqref{det1} is
\begin{multline*}\det\left(\begin{array}{c|c}
\bigl(x_i^{\lambda_j+n-j+\half}\bigr)_{\sss1\le i,j\le n}&\bigl(\x_i^{\lambda_{n+1-j}+j-\half}\bigr)_{\sss1\le i,j\le n}\\[2mm]\hline\rule{0mm}{6mm}
\bigl(\x_i^{\lambda_j+n-j+\half}\bigr)_{\sss1\le i,j\le n}&\bigl(x_i^{\lambda_{n+1-j}+j-\half}\bigr)_{\sss1\le i,j\le n}\end{array}\right)\\*[1mm]
=(-1)^{n(n-1)/2}\,\det\left(\begin{array}{c|c}
\bigl(x_i^{\lambda_j+n-j+\half}\bigr)_{\sss1\le i,j\le n}&\bigl(\x_i^{\lambda_j+n-j+\half}\bigr)_{\sss1\le i,j\le n}\\[2mm]\hline\rule{0mm}{6mm}
\bigl(\x_i^{\lambda_j+n-j+\half}\bigr)_{\sss1\le i,j\le n}&\bigl(x_i^{\lambda_j+n-j+\half}\bigr)_{\sss1\le i,j\le n}\end{array}\right).\end{multline*}

Now note that, for any $n\times n$ matrices $A$ and $B$,
\begin{equation}\label{detid1}
\det\left(\begin{array}{c|c}A&B\\\hline\rule{0mm}{4.4mm}B&A\end{array}\right)=\det\left(\begin{array}{c|c}A-B&B\\\hline\rule{0mm}{4.4mm}B-A&A\end{array}\right)=
\det\left(\begin{array}{c|c}A-B&B\\\hline\rule{0mm}{4.4mm}0&A+B\end{array}\right)=\det(A-B)\det(A+B),\end{equation}
where these equalities are obtained by subtracting the right blocks of the matrix from the left blocks,
adding the top blocks to the bottom blocks, and applying the standard result
that the determinant of a block triangular matrix is the product of the determinants of its diagonal blocks.

Taking $A_{ij}=x_i^{\lambda_j+n-j+\half}$ and $B_{ij}=\x_i^{\lambda_j+n-j+\half}$, for $1\le i,j\le n$, in~\eqref{detid1},
it follows that the numerator of the RHS of~\eqref{det1} is
\begin{equation*}(-1)^{n(n-1)/2}\det_{1\le i,j\le n}\Bigl(x_i^{\lambda_j+n-j+\half}-\x_i^{\lambda_j+n-j+\half}\Bigr)
\;\det_{1\le i,j\le n}\Bigl(x_i^{\lambda_j+n-j+\half}+\x_i^{\lambda_j+n-j+\half}\Bigr).\end{equation*}

Finally, by observing that the denominator of the RHS of~\eqref{det1} is
\begin{equation}\label{den1}\textstyle(-1)^{n(n-1)/2}\prod_{i=1}^n(x_i-\x_i)\,\prod_{1\le i<j\le n}(x_i+\x_i-x_j-\x_j)^2,\end{equation}
and using the odd and even orthogonal character expressions~\eqref{oodef} and~\eqref{oedef}, it follows
that the RHS of~\eqref{det1} is
\begin{equation*}\textstyle
\oo_{(\lambda_1,\ldots,\lambda_n)}(x_1,\ldots,x_n)\;\oe_{(\lambda_1+\half,\ldots,\lambda_n+\half)}(x_1,\ldots,x_n)
\big/\prod_{i=1}^n\bigl(x_i^\half+\x_i^\half\bigr),\end{equation*}
as required.
\end{proof}

\begin{proof}[Proof of Theorem~\ref{thm2}]
Using the Schur polynomial expression \eqref{Schurdef}, the sum of Schur polynomials on the LHS of~\eqref{fact2} is
\begin{multline}\label{det2}
s_{(k_1+\lambda_1,\ldots,k_1+\lambda_n,k_1-\lambda_{n-1},\ldots,k_1-\lambda_0)}(x_1,\ldots,x_n,\x_1,\ldots,\x_n)\\*[-1.5mm]
\qquad\qquad\qquad{}+
s_{(k_2+\lambda_1,\ldots,k_2+\lambda_{n-1},k_2-\lambda_n,\ldots,k_2-\lambda_0)}(x_1,\ldots,x_n,\x_1,\ldots,\x_n)\\*[1mm]
\shoveleft{=\left(\det\left(\begin{array}{c|c}
\bigl(x_i^{k_1+\lambda_j+2n-j}\bigr)_{\begin{subarray}{l}\sss1\le i\le n\\[0.3mm]\sss1\le j\le n\end{subarray}}&
\bigl(x_i^{k_1-\lambda_{n-j}+n-j}\bigr)_{\begin{subarray}{l}\sss1\le i\le n\\[0.3mm]\sss1\le j\le n\end{subarray}}\\[3mm]\hline\rule{0mm}{6mm}
\bigl(\x_i^{k_1+\lambda_j+2n-j}\bigr)_{\begin{subarray}{l}\sss1\le i\le n\\[0.3mm]\sss1\le j\le n\end{subarray}}&
\bigl(\x_i^{k_1-\lambda_{n-j}+n-j}\bigr)_{\begin{subarray}{l}\sss1\le i\le n\\[0.3mm]\sss1\le j\le n\end{subarray}}\end{array}\right)\right.}\\[2mm]
\shoveright{+\left.\det\left(\begin{array}{c|c}
\bigl(x_i^{k_2+\lambda_j+2n-j}\bigr)_{\begin{subarray}{l}\sss1\le i\le n\\[0.3mm]\sss1\le j\le n-1\end{subarray}}&
\bigl(x_i^{k_2-\lambda_{n-j}+n-j}\bigr)_{\begin{subarray}{l}\sss1\le i\le n\\[0.3mm]\sss0\le j\le n\end{subarray}}\\[3mm]\hline\rule{0mm}{6mm}
\bigl(\x_i^{k_2+\lambda_j+2n-j}\bigr)_{\begin{subarray}{l}\sss1\le i\le n\\[0.3mm]\sss1\le j\le n-1\end{subarray}}&
\bigl(\x_i^{k_2-\lambda_{n-j}+n-j}\bigr)_{\begin{subarray}{l}\sss1\le i\le n\\[0.3mm]\sss0\le j\le n\end{subarray}}\end{array}\right)\right)\bigg/
\qquad\qquad}\\
\textstyle\Bigl(\prod_{i=1}^n(x_i-\x_i)\,\prod_{1\le i<j\le n}(x_i-x_j)(\x_i-\x_j)(x_i-\x_j)(x_j-\x_i)\Bigr).\end{multline}

By multiplying row~$i$ in the top blocks of the first matrix in~\eqref{det2} by $\x_i^{k_1+n}$,
row~$i$ in the bottom blocks of the first matrix by $x_i^{k_1+n}$,
row~$i$ in the top blocks of the second matrix by $\x_i^{k_2+n}$, and
row~$i$ in the bottom blocks of the second matrix by $x_i^{k_2+n}$,
for each $i=1,\ldots,n$, it follows that the numerator of the RHS of~\eqref{det2} is
\begin{equation}\label{det2'}\det\left(\begin{array}{c|c}
\bigl(x_i^{\lambda_j+n-j}\bigr)_{\begin{subarray}{l}\sss1\le i\le n\\[0.3mm]\sss1\le j\le n\end{subarray}}&
\bigl(\x_i^{\lambda_{n-j}+j}\bigr)_{\begin{subarray}{l}\sss1\le i\le n\\[0.3mm]\sss1\le j\le n\end{subarray}}\\[3mm]\hline\rule{0mm}{6mm}
\bigl(\x_i^{\lambda_j+n-j}\bigr)_{\begin{subarray}{l}\sss1\le i\le n\\[0.3mm]\sss1\le j\le n\end{subarray}}&
\bigl(x_i^{\lambda_{n-j}+j}\bigr)_{\begin{subarray}{l}\sss1\le i\le n\\[0.3mm]\sss1\le j\le n\end{subarray}}\end{array}\right)
+\det\left(\begin{array}{c|c}
\bigl(x_i^{\lambda_j+n-j}\bigr)_{\begin{subarray}{l}\sss1\le i\le n\\[0.3mm]\sss1\le j\le n-1\end{subarray}}&
\bigl(\x_i^{\lambda_{n-j}+j}\bigr)_{\begin{subarray}{l}\sss1\le i\le n\\[0.3mm]\sss0\le j\le n\end{subarray}}\\[3mm]\hline\rule{0mm}{6mm}
\bigl(\x_i^{\lambda_j+n-j}\bigr)_{\begin{subarray}{l}\sss1\le i\le n\\[0.3mm]\sss1\le j\le n-1\end{subarray}}&
\bigl(x_i^{\lambda_{n-j}+j}\bigr)_{\begin{subarray}{l}\sss1\le i\le n\\[0.3mm]\sss0\le j\le n\end{subarray}}\end{array}\right).\end{equation}

Since the two matrices in~\eqref{det2'} differ by only a single column (specifically, column~$n$), the sum of their
determinants is the single determinant
\begin{equation*}\det\left(\begin{array}{c|c|c}
\bigl(x_i^{\lambda_j+n-j}\bigr)_{\begin{subarray}{l}\sss1\le i\le n\\[0.3mm]\sss1\le j\le n-1\end{subarray}}&
\bigl(x_i^{\lambda_n}+\x_i^{\lambda_n}\bigr)_{\sss1\le i\le n}&
\bigl(\x_i^{\lambda_{n-j}+j}\bigr)_{\begin{subarray}{l}\sss1\le i\le n\\[0.3mm]\sss1\le j\le n\end{subarray}}\\[3mm]\hline\rule{0mm}{6mm}
\bigl(\x_i^{\lambda_j+n-j}\bigr)_{\begin{subarray}{l}\sss1\le i\le n\\[0.3mm]\sss1\le j\le n-1\end{subarray}}&
\bigl(x_i^{\lambda_n}+\x_i^{\lambda_n}\bigr)_{\sss1\le i\le n}&
\bigl(x_i^{\lambda_{n-j}+j}\bigr)_{\begin{subarray}{l}\sss1\le i\le n\\[0.3mm]\sss1\le j\le n\end{subarray}}\end{array}\right).\end{equation*}

By reversing the order of the columns in the rightmost blocks of the matrix, it follows that the previous determinant is
\begin{equation*}
(-1)^{n(n-1)/2}\det\left(\begin{array}{c|c|c|c}
\bigl(x_i^{\lambda_j+n-j}\bigr)_{\begin{subarray}{l}\sss1\le i\le n\\[0.3mm]\sss1\le j\le n-1\end{subarray}}&
\bigl(x_i^{\lambda_n}+\x_i^{\lambda_n}\bigr)_{\sss1\le i\le n}&
\bigl(\x_i^{\lambda_0+n}\bigr)_{\sss1\le i\le n}&
\bigl(\x_i^{\lambda_j+n-j}\bigr)_{\begin{subarray}{l}\sss1\le i\le n\\[0.3mm]\sss1\le j\le n-1\end{subarray}}\\[3mm]\hline\rule{0mm}{6mm}
\bigl(\x_i^{\lambda_j+n-j}\bigr)_{\begin{subarray}{l}\sss1\le i\le n\\[0.3mm]\sss1\le j\le n-1\end{subarray}}&
\bigl(x_i^{\lambda_n}+\x_i^{\lambda_n}\bigr)_{\sss1\le i\le n}&
\bigl(x_i^{\lambda_0+n}\bigr)_{\sss1\le i\le n}&
\bigl(x_i^{\lambda_j+n-j}\bigr)_{\begin{subarray}{l}\sss1\le i\le n\\[0.3mm]\sss1\le j\le n-1\end{subarray}}\end{array}\right).\end{equation*}

It can be seen, by applying standard determinant operations, that for any $n\times(n-1)$ matrices~$A$ and $B$, and $n\times1$ matrices $c$, $d$ and $e$,
\begin{align}\notag\det\left(\begin{array}{c|c|c|c}A&c&d&B\\\hline\rule{0mm}{4.4mm}B&c&e&A\end{array}\right)
&=\det\left(\begin{array}{c|c|c|c}A+B&c&d&B\\\hline\rule{0mm}{4.4mm}A+B&c&e&A\end{array}\right)
=\det\left(\begin{array}{c|c|c|c}A+B&c&d&B\\\hline\rule{0mm}{4.4mm}0&0&e-d&A-B\end{array}\right)\\[1mm]
\label{detid2}&=\det(A+B\,|\,c)\,\det(e-d\,|\,A-B).\end{align}

Taking $A_{ij}=x_i^{\lambda_j+n-j}$, $B_{ij}=\x_i^{\lambda_j+n-j}$, $c_i=x_i^{\lambda_n}+\x_i^{\lambda_n}$,
$d_i=\x_i^{\lambda_0+n}$ and $e_i=x_i^{\lambda_0+n}$, for $1\le i\le n$ and $1\le j\le n-1$, in~\eqref{detid2},
it follows that the numerator of the RHS of~\eqref{det2} is
\begin{equation*}(-1)^{n(n-1)/2}\det_{1\le i,j\le n}\Bigl(x_i^{\lambda_j+n-j}+\x_i^{\lambda_j+n-j}\Bigr)
\;\det_{1\le i,j\le n}\Bigl(x_i^{\lambda_{j-1}+n-j+1}-\x_i^{\lambda_{j-1}+n-j+1}\Bigr).\end{equation*}

Finally, by observing that the denominator of the RHS of~\eqref{det2} is the same as the denominator of the RHS of~\eqref{det1} (and hence
given by~\eqref{den1}),
and using the odd and even orthogonal character expressions~\eqref{oodef} and~\eqref{oedef}, it follows
that the RHS of~\eqref{det2} is
\begin{equation*}\textstyle
(1+\delta_{\lambda_n,0})\,\oo_{(\lambda_0+\half,\ldots,\lambda_{n-1}+\half)}(x_1,\ldots,x_n)\;\oe_{(\lambda_1,\ldots,\lambda_n)}(x_1,\ldots,x_n)
\big/\prod_{i=1}^n\bigl(x_i^\half+\x_i^\half\bigr),\end{equation*}
as required.
\end{proof}

\begin{proof}[Proof of Theorem~\ref{thm3}]
Using the Schur polynomial expression \eqref{Schurdef}, the Schur polynomial on the LHS of~\eqref{fact3} is
\begin{multline}\label{det3}
s_{(k+\lambda_1,\ldots,k+\lambda_n,k-\lambda_n,\ldots,k-\lambda_0)}(x_1,\ldots,x_n,\x_1,\ldots,\x_n,1)\\[1mm]
=\frac{\rule[-13mm]{0mm}{0mm}\det\left(\begin{array}{c|c}
\bigl(x_i^{k+\lambda_j+2n-j+1}\bigr)_{\begin{subarray}{l}\sss1\le i\le n\\[0.3mm]\sss1\le j\le n\end{subarray}}&
\bigl(x_i^{k-\lambda_{n-j}+n-j}\bigr)_{\begin{subarray}{l}\sss1\le i\le n\\[0.3mm]\sss0\le j\le n\end{subarray}}\\[3mm]\hline
\rule{0mm}{6mm}\bigl(\x_i^{k+\lambda_j+2n-j+1}\bigr)_{\begin{subarray}{l}\sss1\le i\le n\\[0.3mm]\sss1\le j\le n\end{subarray}}&
\bigl(\x_i^{k-\lambda_{n-j}+n-j}\bigr)_{\begin{subarray}{l}\sss1\le i\le n\\[0.3mm]\sss0\le j\le n\end{subarray}}\\[3mm]\hline
\rule{0mm}{4.5mm}(1)_{\sss1\le j\le n}&(1)_{\sss0\le j\le n}
\end{array}\right)}
{\rule{0mm}{4mm}\prod_{i=1}^n(x_i-\x_i)(x_i-1)(\x_i-1)\,\prod_{1\le i<j\le n}(x_i-x_j)(\x_i-\x_j)(x_i-\x_j)(x_j-\x_i)}.\end{multline}

By multiplying row $i$ in the top blocks of the matrix in~\eqref{det3} by~$\x_i^{k+n+\half}$
and row~$i$ in the central blocks by $x_i^{k+n+\half}$, for each $i=1,\ldots,n$, and reversing the order of the columns
in the right blocks, it follows that the numerator of the RHS of~\eqref{det3} is
\begin{equation*}(-1)^{n(n+1)/2}\det\left(\begin{array}{c|c}
\bigl(x_i^{\lambda_j+n-j+\half}\bigr)_{\begin{subarray}{l}\sss1\le i\le n\\[0.3mm]\sss1\le j\le n\end{subarray}}&
\bigl(\x_i^{\lambda_j+n-j+\half}\bigr)_{\begin{subarray}{l}\sss1\le i\le n\\[0.3mm]\sss0\le j\le n\end{subarray}}\\[3mm]\hline
\rule{0mm}{6mm}\bigl(\x_i^{\lambda_j+n-j+\half}\bigr)_{\begin{subarray}{l}\sss1\le i\le n\\[0.3mm]\sss1\le j\le n\end{subarray}}&
\bigl(x_i^{\lambda_j+n-j+\half}\bigr)_{\begin{subarray}{l}\sss1\le i\le n\\[0.3mm]\sss0\le j\le n\end{subarray}}\\[3mm]\hline
\rule{0mm}{4.5mm}(1)_{\sss1\le j\le n}&(1)_{\sss0\le j\le n}
\end{array}\right).\end{equation*}

It can be seen, by applying standard determinant operations, that
for any matrices $(A_{ij})_{\begin{subarray}{l}\sss1\le i\le n\\[0.3mm]\sss0\le j\le n\end{subarray}}$
and $(B_{ij})_{\begin{subarray}{l}\sss1\le i\le n\\[0.3mm]\sss0\le j\le n\end{subarray}}$,
\begin{align}\notag\det\left(\begin{array}{c|c}(A_{ij})_{\begin{subarray}{l}\sss1\le i\le n\\[0.3mm]\sss1\le j\le n\end{subarray}}&
(B_{ij})_{\begin{subarray}{l}\sss1\le i\le n\\[0.3mm]\sss0\le j\le n\end{subarray}}\\[3mm]\hline
\rule{0mm}{5mm}(B_{ij})_{\begin{subarray}{l}\sss1\le i\le n\\[0.3mm]\sss1\le j\le n\end{subarray}}&
(A_{ij})_{\begin{subarray}{l}\sss1\le i\le n\\[0.3mm]\sss0\le j\le n\end{subarray}}\\[3mm]\hline
\rule{0mm}{4.5mm}(1)_{\sss1\le j\le n}&(1)_{\sss0\le j\le n}\end{array}\right)
&=\det\left(\begin{array}{c|c}(A_{ij}-B_{ij})_{\begin{subarray}{l}\sss1\le i\le n\\[0.3mm]\sss1\le j\le n\end{subarray}}&
(B_{ij})_{\begin{subarray}{l}\sss1\le i\le n\\[0.3mm]\sss0\le j\le n\end{subarray}}\\[3mm]\hline
\rule{0mm}{5mm}(B_{ij}-A_{ij})_{\begin{subarray}{l}\sss1\le i\le n\\[0.3mm]\sss1\le j\le n\end{subarray}}&
(A_{ij})_{\begin{subarray}{l}\sss1\le i\le n\\[0.3mm]\sss0\le j\le n\end{subarray}}\\[3mm]\hline
\rule{0mm}{4.5mm}(0)_{\sss1\le j\le n}&(1)_{\sss0\le j\le n}\end{array}\right)\\[2mm]
\notag&\hspace{-56mm}=\det\left(\begin{array}{c|c}(A_{ij}-B_{ij})_{\begin{subarray}{l}\sss1\le i\le n\\[0.3mm]\sss1\le j\le n\end{subarray}}&
(B_{ij})_{\begin{subarray}{l}\sss1\le i\le n\\[0.3mm]\sss0\le j\le n\end{subarray}}\\[3mm]\hline
\rule{0mm}{5mm}(0)_{\begin{subarray}{l}\sss1\le i\le n\\[0.3mm]\sss1\le j\le n\end{subarray}}&
(A_{ij}+B_{ij})_{\begin{subarray}{l}\sss1\le i\le n\\[0.3mm]\sss0\le j\le n\end{subarray}}\\[3mm]\hline
\rule{0mm}{4.5mm}(0)_{\sss1\le j\le n}&(1)_{\sss0\le j\le n}\end{array}\right)
=\det_{1\le i,j\le n}\bigl(A_{ij}-B_{ij}\bigr)\,
\det\left(\begin{array}{c}(A_{ij}+B_{ij})_{\begin{subarray}{l}\sss1\le i\le n\\[0.3mm]\sss0\le j\le n\end{subarray}}\\[3mm]\hline
\rule{0mm}{4.5mm}(1)_{\sss0\le j\le n}\end{array}\right)\\[2mm]
\label{detid3}&={\textstyle\frac{1}{2}}\det_{1\le i,j\le n}\bigl(A_{ij}-B_{ij}\bigr)\,
\det\left(\begin{array}{c}(A_{ij}+B_{ij})_{\begin{subarray}{l}\sss1\le i\le n\\[0.3mm]\sss0\le j\le n\end{subarray}}\\[3mm]\hline
\rule{0mm}{4.5mm}(2)_{\sss0\le j\le n}\end{array}\right).\end{align}

Taking $A_{ij}=x_i^{\lambda_j+n-j+\half}$ and $B_{ij}=\x_i^{\lambda_j+n-j+\half}$, for $1\le i\le n$ and $0\le j\le n$,
in~\eqref{detid3},
it follows that the numerator of the RHS of~\eqref{det3} is
\begin{equation*}{\textstyle\frac{1}{2}}\,(-1)^{n(n+1)/2}\det_{1\le i,j\le n}\Bigl(x_i^{\lambda_j+n-j+\half}-\x_i^{\lambda_j+n-j+\half}\Bigr)\,
\det\left(\begin{array}{c}\bigl(x_i^{\lambda_j+n-j+\half}+
\x_i^{\lambda_j+n-j+\half}\bigr)_{\begin{subarray}{l}\sss1\le i\le n\\[0.3mm]\sss0\le j\le n\end{subarray}}\\[3mm]\hline
\rule{0mm}{4.5mm}(2)_{\sss0\le j\le n}\end{array}\right).\end{equation*}

By observing that the denominator of the RHS of~\eqref{det3} is
\begin{equation}\label{den3}\textstyle(-1)^{n(n+1)/2}\,\prod_{i=1}^n\bigl(x_i-\x_i\bigr)\bigl(x_i+\x_i-2\bigr)\,
\prod_{1\le i<j\le n}(x_i+\x_i-x_j-\x_j)^2,\end{equation}
and using the odd and even orthogonal character expressions~\eqref{oodef} and~\eqref{oedef}, it follows
that the RHS of~\eqref{det3} is
\begin{equation*}\textstyle
\oo_{(\lambda_1,\ldots,\lambda_n)}(x_1,\ldots,x_n)\;\oe_{(\lambda_0+\half,\ldots,\lambda_n+\half)}(x_1,\ldots,x_n,1)
\big/\bigl(2\prod_{i=1}^n\bigl(x_i^\half+\x_i^\half\bigr)\bigr),\end{equation*}
as required for the proof of~\eqref{fact3}.

The proof of~\eqref{fact3'} can be obtained by slightly modifying the proof of~\eqref{fact3}. The details will be omitted. 
\end{proof}

\begin{proof}[Proof of Proposition~\ref{prop1}]
Using~\eqref{oedef}, we have
\begin{equation}\label{oeodddet1}(1+\delta_{\lambda_n,0})\,\oe_{(\lambda_0,\ldots,\lambda_n)}(x_1,\ldots,x_n,1)
=\frac{\rule[-7.5mm]{0mm}{0mm}\det\left(\begin{array}{c}\bigl(x_i^{\lambda_j+n-j}+
\x_i^{\lambda_j+n-j}\bigr)_{\begin{subarray}{l}\sss1\le i\le n\\[0.3mm]\sss0\le j\le n\end{subarray}}\\[3mm]\hline
\rule{0mm}{4.5mm}(2)_{\sss0\le j\le n}\end{array}\right)}
{\rule{0mm}{4mm}\prod_{i=1}^n\bigl(x_i+\x_i-2\bigr)\,\prod_{1\le i<j\le n}(x_i+\x_i-x_j-\x_j)}.\end{equation}

By subtracting column $j$ of the matrix from column $j-1$, for each $j=1,\ldots,n$, it follows that the numerator of the RHS 
of~\eqref{oeodddet1} is
\begin{equation*}2\det_{1\le i,j\le n}\Bigl(x_i^{\lambda_{j-1}+n-j+1}+\x_i^{\lambda_{j-1}+n-j+1}-x_i^{\lambda_j+n-j}-\x_i^{\lambda_j+n-j}\Bigr).\end{equation*}

Now let $(\lambda_0,\ldots,\lambda_n)$ be given by~\eqref{lam4}.  Then~$\lambda_j=\lambda_{j-1}-b$, for $1\le j\le n$. Therefore,
the numerator of the RHS of~\eqref{oeodddet1} is
\begin{align*}&2\det_{1\le i,j\le n}\bigl(x_i^{\lambda_{j-1}+n-j+1}+\x_i^{\lambda_{j-1}+n-j+1}-x_i^{\lambda_{j-1}-b+n-j}-\x_i^{\lambda_{j-1}-b+n-j}\bigr)\\
&=2\det_{1\le i,j\le n}\Bigl(\bigl(x_i^{(b+1)/2}-\x_i^{(b+1)/2}\bigr)\bigl(x_i^{\lambda_{j-1}-b/2+n-j+1/2}-\x_i^{\lambda_{j-1}-b/2+n-j+1/2}\bigr)\Bigr)\\
&=2\,{\textstyle\prod_{i=1}^n}\bigl(x_i^{(b+1)/2}-\x_i^{(b+1)/2}\bigr)
\det_{1\le i,j\le n}\bigl(x_i^{\lambda_{j-1}-b/2+n-j+1/2}-\x_i^{\lambda_{j-1}-b/2+n-j+1/2}\bigr)\\
&=2\,{\textstyle\prod_{i=1}^n\bigl(x_i^{\half}-\x_i^{\half}\bigr)\bigl(\sum_{j=-b/2}^{b/2}x_i^j\bigr)}
\det_{1\le i,j\le n}\bigl(x_i^{\lambda_{j-1}-b/2+n-j+1/2}-\x_i^{\lambda_{j-1}-b/2+n-j+1/2}\bigr).\end{align*}

Using~\eqref{oodef} and the denominator of the RHS of~\eqref{oeodddet1}, we now obtain~\eqref{oeodd1}, as required.
\end{proof}

\begin{proof}[Proof of Proposition~\ref{prop2}]
Using~\eqref{oedef}, we have
\begin{equation}\label{oeodddet2}\oe_{(\lambda_0+\half,\ldots,\lambda_n+\half)}(x_1,\ldots,x_n,1)
=\frac{\rule[-7.5mm]{0mm}{0mm}\det\left(\begin{array}{c}\bigl(x_i^{\lambda_j+n-j+\half}+
\x_i^{\lambda_j+n-j+\half}\bigr)_{\begin{subarray}{l}\sss1\le i\le n\\[0.3mm]\sss0\le j\le n\end{subarray}}\\[3mm]\hline
\rule{0mm}{4.5mm}(2)_{\sss0\le j\le n}\end{array}\right)}
{\rule{0mm}{4mm}\prod_{i=1}^n\bigl(x_i+\x_i-2\bigr)\,\prod_{1\le i<j\le n}(x_i+\x_i-x_j-\x_j)}.\end{equation}

By subtracting column $j+1$ of the matrix from column $j-1$, for each $j=1,\ldots,n-1$,
and subtracting column $n$ from column $n-1$, it follows that the numerator of the RHS  of~\eqref{oeodddet2} is
\begin{multline*}2\,\det\Bigl(\bigl(x_i^{\lambda_{j-1}+n-j+3/2}+\x_i^{\lambda_{j-1}+n-j+3/2}-x_i^{\lambda_{j+1}+n-j-\half}-\x_i^{\lambda_{j+1}+n-j-\half}
\bigr)_{\begin{subarray}{l}\sss1\le i\le n\\[0.3mm]\sss1\le j\le n-1\end{subarray}}\:\Big|\\
\bigl(x_i^{\lambda_{n-1}+3/2}+\x_i^{\lambda_{n-1}+3/2}-x_i^{\lambda_n+\half}-\x_i^{\lambda_n+\half}\bigr)_{\sss1\le i\le n}\Bigr).\end{multline*}

Now let $(\lambda_0,\ldots,\lambda_n)$ be given by~\eqref{lam5}. Then~$\lambda_{j+1}=\lambda_{j-1}-b$, for $1\le j\le n-1$,
and $\lambda_n=b-\lambda_{n-1}$. Therefore, the numerator of the RHS of~\eqref{oeodddet2} is
\begin{align*}&2\,\det\Bigl(\bigl(x_i^{\lambda_{j-1}+n-j+3/2}+\x_i^{\lambda_{j-1}+n-j+3/2}-x_i^{\lambda_{j-1}-b+n-j-\half}-\x_i^{\lambda_{j-1}-b+n-j-\half}
\bigr)_{\begin{subarray}{l}\sss1\le i\le n\\[0.3mm]\sss1\le j\le n-1\end{subarray}}\:\Big|\\
&\qquad\qquad\qquad\qquad\qquad\qquad\qquad\qquad
\bigl(x_i^{\lambda_{n-1}+3/2}+\x_i^{\lambda_{n-1}+3/2}-\x_i^{\lambda_{n-1}-b-\half}-x_i^{\lambda_{n-1}-b-\half}\bigr)_{\sss1\le i\le n}\Bigr)\\
&=2\det_{1\le i,j\le n}\bigl(x_i^{\lambda_{j-1}+n-j+3/2}+\x_i^{\lambda_{j-1}+n-j+3/2}-x_i^{\lambda_{j-1}-b+n-j-\half}-\x_i^{\lambda_{j-1}-b+n-j-\half}\bigr)\\
&=2\det_{1\le i,j\le n}\Bigl(\bigl(x_i^{b/2+1}-\x_i^{b/2+1}\bigr)\bigl(x_i^{\lambda_{j-1}-b/2+n-j+\half}-\x_i^{\lambda_{j-1}-b/2+n-j+\half}\bigr)\Bigl)\\
&=2\,{\textstyle\prod_{i=1}^n}\bigl(x_i^{b/2+1}-\x_i^{b/2+1}\bigr)
\det_{1\le i,j\le n}\bigl(x_i^{\lambda_{j-1}-b/2+n-j+\half}-\x_i^{\lambda_{j-1}-b/2+n-j+\half}\bigr)\\
&=2\,{\textstyle\prod_{i=1}^n\bigl(x_i^{\half}-\x_i^{\half}\bigr)\bigl(\sum_{j=-(b+1)/2}^{(b+1)/2}x_i^j\bigr)}
\det_{1\le i,j\le n}\bigl(x_i^{\lambda_{j-1}-b/2+n-j+\half}-\x_i^{\lambda_{j-1}-b/2+n-j+\half}\bigr).\end{align*}

Using~\eqref{oodef} and the denominator of the RHS of~\eqref{oeodddet2}, we now obtain~\eqref{oeodd2}, as required.
\end{proof}

\section{Applications}\label{applications}
In this section, we obtain certain previously-known Schur polynomial factorization identities 
as corollaries of Theorems~\ref{thm1}--\ref{thm3}, 
and we consider some combinatorial aspects of our results.  

Throughout the section, the notation 
\begin{equation*}r^n=\underbrace{r,\ldots,r}_{n\text{ times}}\end{equation*}
will be used, where $r$ is a value of a variable of a character, or an entry of a partition or half-partition.

\subsection{General considerations}\label{generalconsid}
In addition to the determinant expressions~\eqref{Schurdef}--\eqref{oedef} for the 
classical group characters, there exist combinatorial expressions in which each function is given as a weighted sum over certain tableaux,
where the weight of a tableau $T$ 
typically has the form $\prod_{i=1}^nx_i^{k_i(T)}$, for integers or half-integers~$k_i(T)$. Hence, the
characters can be regarded as multivariate generating functions for such tableaux.
The simplest case is that of the Schur polynomial $s_{(\lambda_1,\ldots,\lambda_n)}(x_1,\ldots,x_n)$,
which is a weighted sum over all semistandard Young tableaux of shape $(\lambda_1,\ldots,\lambda_n)$
with entries from $\{1,\ldots,n\}$, where the weight of such a tableau $T$ is $\rule{0mm}{2mm}\prod_{i=1}^n x_i^{\text{number of $i$'s in $T$}}$.
In the cases of symplectic and orthogonal characters,~\eqref{spdef}--\eqref{oedef}, there are several 
different types of tableaux which can be used.
For further information regarding combinatorial expressions for characters, see, for example, Krattenthaler~\cite[Appendix]{Kra98}, Fulmek and 
Krattenthaler~\cite[Sec.~3]{FulKra97}, and Sundaram~\cite{Sun90a}.

It follows that the main results of this paper, Theorems~\ref{thm1}--\ref{thm3}, can be interpreted combinatorially as factorization identities satisfied by 
generating functions for tableaux.  Accordingly, it would be interesting to obtain combinatorial proofs of these identities, as will be done 
in~\cite{AyyFis18}.

If all of the variables $x_1,\ldots,x_n$ in the characters~\eqref{Schurdef}--\eqref{oedef} are set to~1, 
then this gives the numbers of associated tableaux, and the dimensions of associated irreducible representations.  
There exist product formulae for each of these numbers, as follows.
For a partition $(\lambda_1,\ldots,\lambda_n)$,
\begin{equation}\label{numSSYT}s_{(\lambda_1,\ldots,\lambda_n)}(1^n)=\frac{\rule[-2.2mm]{0mm}{0mm}\prod_{1\le i<j\le n}(\lambda_i-\lambda_j-i+j)}
{\prod_{i=1}^{n-1}i!}\end{equation}
and
\begin{multline}\label{numsp}\sp_{(\lambda_1,\ldots,\lambda_n)}(1^n)\\
=\frac{\rule[-2.2mm]{0mm}{0mm}\prod_{i=1}^n(\lambda_i-i+n+1)\prod_{1\le i<j\le n}(\lambda_i-\lambda_j-i+j)(\lambda_i+\lambda_j-i-j+2n+2)}
{\prod_{i=1}^n(2i-1)!},\end{multline}
and for a partition or half-partition $(\lambda_1,\ldots,\lambda_n)$,
\begin{multline}\label{numoo}\oo_{(\lambda_1,\ldots,\lambda_n)}(1^n)\\
=\frac{\rule[-2.2mm]{0mm}{0mm}\prod_{i=1}^n(2\lambda_i-2i+2n+1)
\prod_{1\le i<j\le n}(\lambda_i-\lambda_j-i+j)(\lambda_i+\lambda_j-i-j+2n+1)}{\prod_{i=1}^n(2i-1)!}\end{multline}
and
\begin{equation}\label{numoe}\oe_{(\lambda_1,\ldots,\lambda_n)}(1^n)=
\frac{\rule[-2.2mm]{0mm}{0mm}2^n\prod_{1\le i<j\le n}(\lambda_i-\lambda_j-i+j)(\lambda_i+\lambda_j-i-j+2n)}
{(1+\delta_{\lambda_n,0})\prod_{i=1}^{n-1}(2i)!}.\end{equation}

By combining~\eqref{oeshifted} and~\eqref{numoe}, it follows that,
for a partition or half-partition $(\lambda_1,\ldots,\lambda_n)$,
\begin{equation}\label{numoominus}\oo_{(\lambda_1,\ldots,\lambda_n)}((-1)^n)=
\frac{\rule[-2.2mm]{0mm}{0mm}(-1)^{\sum_{i=1}^n\lambda_i}\prod_{1\le i<j\le n}(\lambda_i-\lambda_j-i+j)(\lambda_i+\lambda_j-i-j+2n+1)}
{\prod_{i=1}^{n-1}(2i)!}.\end{equation}

For further information regarding~\eqref{numSSYT}--\eqref{numoe}, and derivations using 
Weyl's denominator formula, see, for example, Fulton and Harris~\cite[Ch.~24]{FulHar91}.

We now proceed to the consideration of plane partitions and related combinatorial objects (such as rhombus tilings), 
and alternating sign matrices and related combinatorial objects (such as alternating sign triangles). 
For definitions and information regarding plane partitions, alternating sign matrices, and symmetry classes of these
objects, see, for example, Bressoud~\cite{Bre99}, 
Krattenthaler~\cite{Kra16}, and Behrend, Fischer and Konvalinka~\cite[Secs. 1.2 \& 1.3]{BehFisKon17}.
Of particular relevance here are the facts that the sizes of certain sets of these objects
are given by product formulae, and that in all of the cases which will be considered,
these product formulae can be related to cases of the product formulae~\eqref{numSSYT}--\eqref{numoominus} 
in which the partition is of rectangular or double-staircase shape.  The rectangular shapes are associated with
exact equalities between numbers of plane partitions and specializations of~\eqref{numSSYT}--\eqref{numoominus} (as will be seen
in the list in Subsection~\ref{rect}),
while the double-staircase shapes are associated with equalities, up to certain
simple prefactors (which are typically powers of~3), between numbers of plane partitions or alternating sign matrices
and specializations of~\eqref{numSSYT}--\eqref{numoominus} (as will be seen in the list in Subsection~\ref{doubstair}). 
In the rectangular cases, certain combinatorial connections can be found between the plane partitions 
and tableaux enumerated by the same formula. The simplest example is that of equality between the number of 
plane partitions in an $a\times b\times c$ box and the number of semistandard Young tableaux of 
shape $(a^b,0^c)$ with entries from $\{1,\ldots,b+c\}$, for which there is an almost-trivial bijection between the respective sets of objects.
On the other hand, no such combinatorial connections are currently known for any of the double-staircase cases. 
Furthermore, in some of these cases, there are several
different types of plane partitions, alternating sign matrices or related objects which are enumerated by the same product formula,
but without a bijective explanation for the equalities currently known.
We note that the explicit appearance of characters of classical groups in the enumeration of alternating sign matrices was first obtained, algebraically,
by Okada~\cite{Oka06}.

In the next two subsections, the partitions in the general theorems of Section~\ref{results} will first be specialized 
to have rectangular shape (Subsection~\ref{rect}) or double-staircase shape (Subsection~\ref{doubstair}), which will lead to 
factorization identities of Okada~\cite{Oka98}, Ciucu and Krattenthaler~\cite{CiuKra09},
Behrend, Fischer and Konvalinka~\cite{BehFisKon17}, and
Ayyer, Behrend and Fischer~\cite{AyyBehFis16}.  Then, 
all of the variables in these results will be set to~1, which, using the relations discussed above, 
will lead to factorization identities for numbers of plane partitions, alternating sign matrices or related objects.
Almost all of these identities have appeared previously in the literature, at least in some form, and although they can easily be verified directly 
using the product formulae 
for the relevant objects, it seems interesting that they can also be regarded as specializations of multivariate factorization
results.  

\subsection{Partitions of rectangular shape}\label{rect}
In the results in this subsection,~$m$ is a nonnegative integer.

The following result, which is a corollary of Theorem~\ref{thm1}, was 
noted without proof by Okada~\cite[Lem.~5.2, 1st Eq.]{Oka98}, and proved by 
Ciucu and Krattenthaler~\cite[Thms.~3.1 \& 3.2]{CiuKra09}.
\begin{corollary}
We have
\begin{align}\label{rect1}
s_{((2m)^n,0^n)}(x_1,\ldots,x_n,\x_1,\ldots,\x_n)&=(-1)^{mn}\,\oo_{(m^n)}(x_1,\ldots,x_n)\;\oo_{(m^n)}(-x_1,\ldots,-x_n)
\intertext{and}
\label{rect2}s_{((2m+1)^n,0^n)}(x_1,\ldots,x_n,\x_1,\ldots,\x_n)&=\sp_{(m^n)}(x_1,\ldots,x_n)\;\oe_{((m+1)^n)}(x_1,\ldots,x_n).\end{align}
\end{corollary}
\begin{proof}
Taking $(\lambda_1,\ldots,\lambda_n)=(m^n)$ in~\eqref{fact11} and~\eqref{fact12} gives~\eqref{rect1}
and~\eqref{rect2}, respectively.
\end{proof}

The following result, which is a corollary of Theorem~\ref{thm2}, was obtained by Ciucu and Krattenthaler~\cite[Thms.~3.3 \& 3.4]{CiuKra09}.
\begin{corollary}
We have
\begin{align}\notag&
s_{((2m)^n,0^n)}(x_1,\ldots,x_n,\x_1,\ldots,\x_n)+s_{((2m)^{n-1},0^{n+1})}(x_1,\ldots,x_n,\x_1,\ldots,\x_n)\\*[-1mm]
\label{addrect1}&\hspace{62mm}=(1+\delta_{m,0})\,\sp_{(m^n)}(x_1,\ldots,x_n)\;\oe_{(m^n)}(x_1,\ldots,x_n)
\intertext{and}
\notag&s_{((2m+1)^n,0^n)}(x_1,\ldots,x_n,\x_1,\ldots,\x_n)+s_{((2m+1)^{n-1},0^{n+1})}(x_1,\ldots,x_n,\x_1,\ldots,\x_n)\\*[-1mm]
\label{addrect2}&\hspace{62mm}=(-1)^{mn}\,\oo_{((m+1)^n)}(x_1,\ldots,x_n)\;\oo_{(m^n)}(-x_1,\ldots,-x_n).\end{align}
\end{corollary}
\begin{proof}
Taking $(\lambda_0,\ldots,\lambda_n)=(m^{n+1})$ in~\eqref{fact21} and~\eqref{fact22} gives~\eqref{addrect1}
and~\eqref{addrect2}, respectively.
\end{proof}

The following result, which is a corollary of Theorem~\ref{thm3} and Proposition~\ref{prop1},
was noted without proof by Okada~\cite[Lem.~5.2, 2nd Eq.]{Oka98}.
\begin{corollary}
We have
\begin{align}\label{rect3}
s_{((2m)^n,0^{n+1})}(x_1,\ldots,x_n,\x_1,\ldots,\x_n,1)&=\sp_{(m^n)}(x_1,\ldots,x_n)\;\oo_{(m^n)}(x_1,\ldots,x_n)
\intertext{and}
\label{rect4}s_{((2m+1)^n,0^{n+1})}(x_1,\ldots,x_n,\x_1,\ldots,\x_n,1)&=\sp_{(m^n)}(x_1,\ldots,x_n)\;\oo_{((m+1)^n)}(x_1,\ldots,x_n).\end{align}
\end{corollary}
\begin{proof}
To obtain~\eqref{rect3}, take $(\lambda_0,\ldots,\lambda_n)=(m^{n+1})$ and $k=m$ in the $-$ case of~\eqref{fact3'},
and apply~\eqref{ooshifted} to the odd orthogonal character.
Then take $a=m$ and $b=0$ in~\eqref{lam4}, and apply~\eqref{oeodd1} to the even orthogonal character.  
To obtain~\eqref{rect4}, take $(\lambda_0,\ldots,\lambda_n)=\bigl((m+\frac{1}{2})^{n+1}\bigr)$ and $k=m+\frac{1}{2}$ in~\eqref{fact3}, 
and apply~\eqref{ooshifted} to the odd orthogonal character.
Then take $a=m+1$ and $b=0$ in~\eqref{lam4}, and apply~\eqref{oeodd1} to the even orthogonal character.  
\end{proof}

We now set $x_i=1$, for each~$i$, in~\eqref{rect1}--\eqref{rect4}.  The terms in the resulting equations 
can then be identified as follows, for appropriate $a$, $b$ and $c$.
\begin{list}{$\bullet$}{\setlength{\topsep}{0.8mm}\setlength{\labelwidth}{2mm}\setlength{\leftmargin}{6mm}}
\item We have 
\begin{equation*}\quad s_{(a^b,0^c)}(1^{b+c})=\prod_{i=1}^a\prod_{j=1}^b\prod_{k=1}^c\frac{i+j+k-1}{i+j+k-2},\end{equation*}
which is the number of plane partitions in an $a\times b\times c$ box (MacMahon~\cite[Sec.~429]{Mac16}). 
We denote this number as $\mathrm{PP}(a,b,c)$.  
\item We have 
\begin{equation*}\quad(-1)^{ab}\oo_{(a^b)}((-1)^b)=\sp_{(a^{b-1})}(1^{b-1})=\prod_{1\le i<j\le b}\frac{i+j+2a-1}{i+j-1},\end{equation*}
which is the number of transpose complementary plane partitions in a $(2a)\times b\times b$ box (Proctor~\cite{Pro86,Pro88}). 
We denote this number as $\mathrm{TCPP}(2a,b,b)$.
\item We have 
\begin{equation*}\quad\oo_{(a^b)}(1^b)=\prod_{1\le i\le j\le b}\frac{i+j+2a-1}{i+j-1},\end{equation*}
which is the number of symmetric plane partitions in a $(2a)\times b\times b$ box (Andrews~\cite{And78}).
We denote this number as $\mathrm{SPP}(2a,b,b)$.  Note that $\mathrm{SPP}(2a,b,b)$ is also the number of certain weighted rhombus tilings of a half-hexagon,
as shown by Ciucu~\cite[Sec.~6]{Ciu97}. (However, no bijective explanation for this is currently known.)
More specifically, using the notation of Ciucu~\cite[Eq.~(6.1)]{Ciu97}, $\mathrm{SPP}(2a,b,b)=2^bM(H(b,b,2a)^-)$.
\item We have 
\begin{equation*}\quad\oe_{(a^b)}(1^b)=2\prod_{1\le i<j\le b}\frac{i+j+2a-2}{i+j-2},\end{equation*}
which is the number of certain restricted symmetric plane partitions in a $(2a)\times b\times b$ box
(Ciucu and Krattenthaler~\cite[pp.~54--56]{CiuKra09}). We denote this number as $\mathrm{SPP}^\ast(2a,b,b)$.
Note that $\mathrm{SPP}^\ast(2a,b,b)=2\,\mathrm{SPP}(2a,b-1,b-1)$.
\end{list}

By using a simple bijection between plane partitions and rhombus tilings, each of the previous cases
can be stated in terms of the latter.  For example, $\mathrm{PP}(a,b,c)$ is the number of rhombus
tilings of a hexagon with side lengths $a$, $b$, $c$, $a$, $b$, $c$, $\mathrm{TCPP}(2a,b,b)$ is 
the number of rhombus tilings of a hexagon with side lengths $2a$, $b$, $b$, $2a$, $b$, $b$ which 
are symmetric with respect to the line bisecting the two sides of length $2a$, and
$\mathrm{SPP}(2a,b,b)$ is the number of rhombus tilings of a hexagon with side lengths $2a$, $b$, $b$, $2a$, $b$, $b$ which 
are symmetric with respect to the line joining the two points where sides of length $b$ meet.

Some additional remarks regarding the identifications in the previous list, and the identifications which will appear
in the analogous list in Subsection~\ref{doubstair}, are as follows.
The initial equalities (up to prefactors) between characters with all variables set to~$\pm1$ and 
explicit product formulae can be obtained straightforwardly using~\eqref{numSSYT}--\eqref{numoominus}.  
The further equalities between explicit product formulae and numbers of plane partitions, alternating sign matrices or related objects 
are fundamental results in the enumeration of these objects and their symmetry classes.
The references given in parentheses are for the first proofs of these results.   
However, in most cases, several later proofs are also known, 
and in some cases, the result was conjectured in the literature long before being proved.  The references for the later proofs and conjectures are
omitted here.  Also note that the explicit product formulae often appear in the literature in several different forms, but that
the equality between the forms can be obtained straightforwardly.

By applying identifications in the previous list, the results~\eqref{rect1}--\eqref{rect4} with all variables set to~$1$ give the following
factorization identities for numbers of plane partitions:
\begin{align}
\label{PP1}\mathrm{PP}(2m,n,n)&=\mathrm{SPP}(2m,n,n)\,\mathrm{TCPP}(2m,n,n),\\
\mathrm{PP}(2m+1,n,n)&=\mathrm{SPP}^\ast(2m+2,n,n)\,\mathrm{TCPP}(2m,n+1,n+1),\\
\notag\mathrm{PP}(2m+1,n,n)+\mathrm{PP}(2m+1,n-1,n+1)\hspace*{-38mm}&\\
&=\mathrm{SPP}(2m+2,n,n)\,\mathrm{TCPP}(2m,n,n),\\
\notag\mathrm{PP}(2m,n,n)+\mathrm{PP}(2m,n-1,n+1)\hspace*{-38mm}&\\
\label{PP4}&=\mathrm{SPP}^\ast(2m,n,n)\,\mathrm{TCPP}(2m,n+1,n+1),\\
\mathrm{PP}(2m,n,n+1)&=\mathrm{SPP}(2m,n,n)\,\mathrm{TCPP}(2m,n+1,n+1),\\
\mathrm{PP}(2m+1,n,n+1)&=\mathrm{SPP}(2m+2,n,n)\,\mathrm{TCPP}(2m,n+1,n+1).
\end{align}

Note that~\eqref{PP1}--\eqref{PP4} were previously obtained from~\eqref{rect1}--\eqref{addrect2} by 
Ciucu and Krattenthaler~\cite[Eqs.~(3.26), (3.29), (3.36), (3.37)]{CiuKra09}.  
For a recent generalization of~\eqref{PP1}, see Ciucu and Krattenthaler~\cite[Thm.~2.1]{CiuKra17}.

\subsection{Partitions of double-staircase shape}\label{doubstair}
The following result, which is a corollary of Theorem~\ref{thm1}, was noted without proof by Ayyer, Behrend and Fischer~\cite[Remark~6.4]{AyyBehFis16}.
\begin{corollary}
We have
\begin{align}\notag&s_{(2n,2n-1,2n-1,2n-2,2n-2,\ldots,2,2,1,1,0)}(x_1,\ldots,x_{2n},\x_1,\ldots,\x_{2n})\\*[-0.5mm]
\notag&\hspace{12mm}=(-1)^n\,\oo_{(n,n-1,n-1,n-2,n-2,\ldots,2,2,1,1,0)}(x_1,\ldots,x_{2n})\\*[-0.5mm]
\label{QASTfact1}&\hspace{77mm}\times\oo_{(n,n-1,n-1,n-2,n-2,\ldots,2,2,1,1,0)}(-x_1,\ldots,-x_{2n})
\intertext{and}
\notag&s_{(2n+1,2n,2n,2n-1,2n-1,\ldots,2,2,1,1,0)}(x_1,\ldots,x_{2n+1},\x_1,\ldots,\x_{2n+1})\\*[-1mm]
\label{QASTfact2}&\hspace{12mm}=\sp_{(n,n-1,n-1,n-2,n-2,\ldots,1,1,0,0)}(x_1,\ldots,x_{2n+1})\;
\oe_{(n+1,n,n,n-1,n-1,\ldots,2,2,1,1)}(x_1,\ldots,x_{2n+1}).\end{align}
\end{corollary}
\begin{proof}Replacing $n$ by $2n$ in~\eqref{fact11}, and then taking $(\lambda_1,\ldots,\lambda_{2n})=
(n,n-1,n-1,n-2,n-2,\ldots,2,2,1,1,0)$ in that equation gives~\eqref{QASTfact1}.
Replacing $n$ by $2n+1$ in~\eqref{fact12}, and then taking $(\lambda_1,\ldots,\lambda_{2n+1})=
(n,n-1,n-1,n-2,n-2,\ldots,1,1,0,0)$ in that equation gives~\eqref{QASTfact2}.\end{proof}

The following result, which is a corollary of Theorem~\ref{thm1}, was also noted 
without proof by Ayyer, Behrend and Fischer~\cite[Remark~5.4]{AyyBehFis16}.
\begin{corollary}
We have
\begin{align}\notag&s_{(2n,2n,2n-1,2n-1,\ldots,1,1,0,0)}(x_1,\ldots,x_{2n+1},\x_1,\ldots,\x_{2n+1})\\*[-1mm]
\label{ASTfact1}&\hspace{17mm}=\oo_{(n,n,n-1,n-1,\ldots,2,2,1,1,0)}(x_1,\ldots,x_{2n+1})\;\oo_{(n,n,n-1,n-1,\ldots,2,2,1,1,0)}(-x_1,\ldots,-x_{2n+1})
\intertext{and}
\notag&s_{(2n-1,2n-1,2n-2,2n-2,\ldots,1,1,0,0)}(x_1,\ldots,x_{2n},\x_1,\ldots,\x_{2n})\\*[-1mm]
\label{ASTfact2}&\hspace{17mm}=\sp_{(n-1,n-1,n-2,n-2,\ldots,1,1,0,0)}(x_1,\ldots,x_{2n})\;\oe_{(n,n,n-1,n-1,\ldots,2,2,1,1)}(x_1,\ldots,x_{2n}).\end{align}
\end{corollary}
\begin{proof}
Replacing $n$ by $2n+1$ in~\eqref{fact11}, and then taking $(\lambda_1,\ldots,\lambda_{2n+1})=
(n,n,n-1,n-1,\ldots,2,2,1,1,0)$ in that equation gives~\eqref{ASTfact1}.
Replacing $n$ by $2n$ in~\eqref{fact12}, and then taking $(\lambda_1,\ldots,\lambda_{2n})=
(n-1,n-1,n-2,n-2,\ldots,1,1,0,0)$ in that equation gives~\eqref{ASTfact2}.
\end{proof}

The following result, which is a corollary of Theorem~\ref{thm3} and Proposition~\ref{prop2},
was noted without proof by Behrend, Fischer and Konvalinka~\cite[Eq.~(63)]{BehFisKon17}.
A form of this result in which the odd orthogonal characters on the RHSs of~\eqref{DASASMfact1} and~\eqref{DASASMfact2} are replaced by
certain symplectic characters was obtained previously by Zinn-Justin~\cite[Eq.~(54)]{Zin13}.
\begin{corollary}
We have
\begin{align}\notag&s_{(2n,2n-1,2n-1,2n-2,2n-2,\ldots,1,1,0,0)}(x_1,\ldots,x_{2n},\x_1,\ldots,\x_{2n},1)=\prod_{i=1}^{2n}(\x_i+1+x_i)\\*
\label{DASASMfact1}&\hspace{15mm}\times
\oo_{(n,n-1,n-1,n-2,n-2,\ldots,2,2,1,1,0)}(x_1,\ldots,x_{2n})\;\sp_{(n-1,n-1,n-2,n-2,\ldots,1,1,0,0)}(x_1,\ldots,x_{2n})
\intertext{and}
\notag&s_{(2n+1,2n,2n,2n-1,2n-1,\ldots,1,1,0,0)}(x_1,\ldots,x_{2n+1},\x_1,\ldots,\x_{2n+1},1)=\prod_{i=1}^{2n+1}(\x_i+1+x_i)\\*
\label{DASASMfact2}&\hspace{15mm}\times
\oo_{(n,n,n-1,n-1,\ldots,2,2,1,1,0)}(x_1,\ldots,x_{2n+1})\;\sp_{(n,n-1,n-1,n-2,n-2,\ldots,1,1,0,0)}(x_1,\ldots,x_{2n+1}).\end{align}
\end{corollary}
\begin{proof}
To obtain~\eqref{DASASMfact1}, replace $n$ by~$2n$ in~\eqref{fact3},~\eqref{lam5} and~\eqref{oeodd2}, and 
let $(\lambda_0,\ldots,\lambda_{2n})$ be given by~\eqref{lam5} with $a=0$ and $b=1$.
Then use~\eqref{fact3} and~\eqref{oeodd2}, with $k=n$, and apply~\eqref{ooshifted} to the odd orthogonal character in~\eqref{oeodd2}.

To obtain~\eqref{DASASMfact2}, replace $n$ by $2n+1$ in~\eqref{fact3},~\eqref{lam5} and~\eqref{oeodd2}, and 
let $(\lambda_0,\ldots,\lambda_{2n+1})$ be given by~\eqref{lam5} with $a=0$ and $b=1$.
Then use~\eqref{fact3} and~\eqref{oeodd2}, with $k=n+1$, and apply~\eqref{ooshifted} to the odd orthogonal character in~\eqref{oeodd2}.
Finally, use~\eqref{Schurrecip} to obtain
$s_{(2n+1,2n+1,2n,2n,\ldots,2,2,1,1,0)}(x_1,\ldots,x_{2n+1},\allowbreak\x_1,\ldots,\x_{2n+1},1)=
s_{(2n+1,2n,2n,2n-1,2n-1,\ldots,1,1,0,0)}(x_1,\ldots,x_{2n+1},\x_1,\ldots,\x_{2n+1},1).$\end{proof}

Note that several of the characters in~\eqref{QASTfact1}--\eqref{DASASMfact2} are related to 
partition functions of certain cases of the six-vertex model.  For further details, see 
Ayyer, Behrend and Fischer~\cite[Thms.~6.3 \&~5.3]{AyyBehFis16},
Behrend, Fischer and Konvalinka~\cite[Cor.~4]{BehFisKon17},
Okada~\cite[Thms.~2.4 \&~2.5]{Oka06}, 
Razumov and Stroganov~\cite[Thms.~2,~3 \&~5]{RazStr04b},~\cite[Eqs.~(28) \&~(31)]{RazStr06a} and
Stroganov~\cite[Eqs.~(6) \&~(11)]{Str04},~\cite[Eq.~(17)]{Str06}.

We now set $x_i=1$, for each~$i$, in~\eqref{ASTfact1}--\eqref{DASASMfact2}.  The terms in the resulting equations 
can then be identified as follows, for appropriate $m$ (where the remarks 
made regarding the list in Subsection~\ref{rect} again apply).
\begin{list}{$\bullet$}{\setlength{\topsep}{0.8mm}\setlength{\labelwidth}{2mm}\setlength{\leftmargin}{6mm}}
\item We have 
\begin{equation*}\quad3^{-m(m-1)/2}\,s_{(m,m-1,m-1,m-2,m-2,\ldots,2,2,1,1,0)}(1^{2m})=\prod_{i=0}^{m-1}\frac{(3i+2)(3i)!}{(m+i)!},\end{equation*}
which is the number of cyclically symmetric plane partitions in an $m\times m\times m$ box (Andrews~\cite{And79,And80}),
and the number of quasi alternating sign triangles with $m$ rows (Ayyer, Behrend and Fischer~\cite{AyyBehFis16}). 
We denote this number as $\mathrm{CSPP}(m)$.
\item We have 
\begin{equation*}\quad3^{-m(m-1)/2}\,s_{(m-1,m-1,m-2,m-2,\ldots,1,1,0,0)}(1^{2m})=\prod_{i=0}^{m-1}\frac{(3i+1)!}{(m+i)!},\end{equation*} 
which is the number of each of the following: $m\times m$ alternating sign matrices (Zeilberger~\cite{Zei96a}, Kuperberg~\cite{Kup96}), order~$m$ descending 
plane partitions (Andrews~\cite{And79,And80}), 
totally symmetric self-complementary plane partitions in a $(2m)\times(2m)\times(2m)$ box (Andrews~\cite{And94}),
and alternating sign triangles with $m$ rows (Ayyer, Behrend and Fischer~\cite{AyyBehFis16}). We denote this number as $\mathrm{ASM}(m)$.
\item We have
\begin{equation*}\quad3^{-m(m-1)/2}\,s_{(m,m-1,m-1,m-2,m-2,\ldots,2,2,1,1,0,0)}(1^{2m+1})=\prod_{i=0}^m\frac{(3i)!}{(m+i)!},\end{equation*} 
which is the number of $(2m+1)\times(2m+1)$ diagonally and antidiagonally symmetric
alternating sign matrices (Behrend, Fischer and Konvalinka~\cite{BehFisKon17}). We denote this number as $\mathrm{DASASM}(2m+1)$.
\item We have
\begin{multline*}\quad\;3^{-m(m-1)}\oo_{(m,m-1,m-1,m-2,m-2,\ldots,2,2,1,1,0)}(1^{2m})\\*
={\textstyle\frac{1}{2}}\,3^{-m(m+1)}\oe_{(m+1,m,m,m-1,m-1,\ldots,2,2,1,1)}(1^{2m+1})
=\prod_{i=1}^m\frac{(6i-1)\,(6i-3)!}{(2i-1)\,(2m+2i-1)!},\end{multline*}
which is the number of totally symmetric plane partitions in a $(2m)\times(2m)\times(2m)$ box (Stembridge~\cite{Ste95}). 
We denote this number as $\mathrm{TSPP}(2m,2m,2m)$. Note that $\mathrm{TSPP}(2m,2m,2m)$ is also the number of certain weighted rhombus tilings of a
pentagonal region, as shown by Ciucu and Krattenthaler~\cite[Sec.~4]{CiuKra00}.  
More specifically, using notation of Ciucu and Krattenthaler~\cite[Eqs.~(4.2)--(4.5) \& Fig.~4.2]{CiuKra00},
$\mathrm{TSPP}(2m,2m,2m)=2^{2m}L(R_1^-)|_{n=m,x=0}$, or using notation of 
Lai and Rohatgi~\cite[Eq.~(4.4) \& Fig.~4.1(d)]{LaiRoh18}, $\mathrm{TSPP}(2m,2m,2m)=2^{2m}\,\mathrm{M}(\rule{0mm}{3mm}^\ast_\ast\mathcal{G}_{m,1})$.
\item We have
\begin{multline*}\quad\;3^{-(m-1)^2}\oo_{(m-1,m-1,m-2,m-2,\ldots,2,2,1,1,0)}(1^{2m-1})\\
={\textstyle\frac{1}{2}}\,3^{-m^2}\oe_{(m,m,m-1,m-1,\ldots,2,2,1,1)}(1^{2m})
=\prod_{i=0}^{m-1}\frac{(6i+1)!}{(2m+2i-1)!},\end{multline*} 
which is the number of certain weighted rhombus tilings 
of a pentagonal region (Ciucu and Krattenthaler~\cite[Sec.~4]{CiuKra00}).  We denote this number as $\mathrm{R}(2m)$.
More specifically, using notation of Ciucu and Krattenthaler~\cite[Eqs.~(4.2)--(4.5) \& Fig.~4.2]{CiuKra00},
$\mathrm{R}(2m)=2^{2m-2}L(R_1^-)|_{n=m-1,x=1}$, or using notation of 
Lai and Rohatgi~\cite[Eq.~(4.4) \& Fig.~4.1(d)]{LaiRoh18}, $\mathrm{R}(2m)=2^{2m-2}\,\mathrm{M}(\rule{0mm}{3mm}^\ast_\ast\mathcal{G}_{m-1,2})$.
\item We have 
\begin{multline*}\quad\;(-1)^m\,3^{-m^2}\oo_{(m,m-1,m-1,m-2,m-2,\ldots,2,2,1,1,0)}((-1)^{2m})\\
=3^{-(m-1)^2}\sp_{(m-1,m-2,m-2,m-3,m-3,\ldots,1,1,0,0)}(1^{2m-1})
=\prod_{i=0}^{m-1}\frac{(2i+1)\,(6i+2)!}{(6i+1)\,(2m+2i)!},\end{multline*}
which is the number of cyclically symmetric transpose complementary plane partitions in a $(2m)\times(2m)\times(2m)$ box
(Mills, Robbins and Rumsey~\cite{MilRobRum87}), and the number of
vertically symmetric quasi alternating sign triangles with $2m$ 
rows (Behrend and Fischer~\cite{BehFis17c}). We denote this number as $\mathrm{CSTCPP}(2m,2m,2m)$.
\item We have 
\begin{multline*}\quad\;3^{-m(m+1)}\oo_{(m,m,m-1,m-1,\ldots,2,2,1,1,0)}((-1)^{2m+1})\\
=3^{-m(m-1)}\sp_{(m-1,m-1,m-2,m-2,\ldots,2,2,1,1,0,0)}(1^{2m})
=\prod_{i=1}^{m}\frac{(6i-2)!}{(2m+2i)!},\end{multline*} 
which is the number of each of the following: $(2m+1)\times(2m+1)$ vertically symmetric alternating sign matrices (Kuperberg~\cite{Kup02}),
order $2m+1$ descending plane partitions invariant under a certain
involution~$\tau$ of Mills, Robbins and Rumsey~\cite[pp.~351--353]{MilRobRum83}
(Mills, Robbins and Rumsey~\cite{MilRobRum87}),
totally symmetric self-complementary plane partitions in a $(4m+2)\times(4m+2)\times(4m+2)$ box invariant under a certain
involution~$\gamma$ of Mills, Robbins and Rumsey~\cite[p.~286]{MilRobRum86} (Ishikawa~\cite{Ish06b}),
$(2m)\times(2m)$ off-diagonally symmetric alternating sign matrices (Kuperberg~\cite{Kup02}),
and vertically symmetric alternating sign triangles with $2m+1$ rows (Behrend and Fischer~\cite{BehFis17c}).
We denote this number as $\mathrm{VSASM}(2m+1)$.
\end{list}

As indicated in Subsection~\ref{generalconsid}, no 
bijective explanations are currently known for any of the 
equalities between numbers of objects in the previous list.

On the other hand, by using a simple bijection between plane partitions and rhombus tilings, several cases 
can be stated in terms of the latter, as follows. 
Cyclically symmetric plane partitions in an $m\times m\times m$ box are in bijection
with rhombus tilings of a regular hexagon with side length~$m$
which are invariant under rotation by $120^\circ$.
Order~$m$ descending plane partitions are in bijection with
rhombus tilings of a hexagon with alternating sides of lengths $m-1$ and $m+1$ 
and a central equilateral triangular hole with side length~2, which are invariant under rotation by $120^\circ$.
See Krattenthaler~\cite{Kra06}.
Totally symmetric self-complementary plane partitions in a $(2m)\times(2m)\times(2m)$ box
are in bijection with rhombus tilings of a regular hexagon with side length~$2m$ which
are invariant under all symmetry operations (i.e., rotation by $60^\circ$
and reflection in all six symmetry axes).
Totally symmetric plane partitions in a $(2m)\times(2m)\times(2m)$ box
are in bijection with rhombus tilings of a regular hexagon with side length~$2m$
which are invariant under rotation by $120^\circ$ and reflection in any symmetry axis
passing through two vertices.
Cyclically symmetric transpose complementary plane partitions in a $(2m)\times(2m)\times(2m)$ box
are in bijection with rhombus tilings of a regular hexagon with side length~$2m$
which are invariant under rotation by $120^\circ$ and reflection in any symmetry axis
bisecting two sides.
Order $2m+1$ descending plane partitions invariant under the involution $\tau$ defined
by Mills, Robbins and Rumsey~\cite[pp.~351--353]{MilRobRum83} are in bijection with
rhombus tilings of a hexagon with alternating sides of lengths~$2m$ and $2m+2$ 
and a central equilateral triangular hole with side length~2, which are invariant under rotation by~$120^\circ$
and reflection in any symmetry axis (i.e., a line bisecting two sides). See Krattenthaler~\cite{Kra06}.

By applying the identifications in the previous list, the results~\eqref{QASTfact1}--\eqref{DASASMfact2} with all variables set to~$1$ give the following
factorization identities for numbers of plane partitions, alternating sign matrices and related objects:
\begin{align}
\label{doubstair1}\mathrm{CSPP}(2n,2n,2n)&=\mathrm{TSPP}(2n,2n,2n)\,\mathrm{CSTCPP}(2n,2n,2n),\\
\label{doubstair2}\mathrm{CSPP}(2n+1,2n+1,2n+1)&=2\,\mathrm{TSPP}(2n,2n,2n)\,\mathrm{CSTCPP}(2n+2,2n+2,2n+2),\\
\label{doubstair3}\mathrm{ASM}(2n+1)&=\mathrm{R}(2n+2)\,\mathrm{VSASM}(2n+1),\\
\label{doubstair4}\mathrm{ASM}(2n)&=2\,\mathrm{R}(2n)\,\mathrm{VSASM}(2n+1),\\
\mathrm{DASASM}(4n+1)&=3^n\,\mathrm{TSPP}(2n,2n,2n)\,\mathrm{VSASM}(2n+1),\\
\mathrm{DASASM}(4n+3)&=3^{n+1}\,\mathrm{R}(2n+2)\,\mathrm{CSTCPP}(2n+2,2n+2,2n+2).
\end{align}

Note that~\eqref{doubstair1} and~\eqref{doubstair2} are closely related to results of
Mills, Robbins and Rumsey~\cite[Thm.~5 with $x=1$, $\mu=0$]{MilRobRum87},~\cite[Thm.~2.1 with $x=1$, $\mu=0$]{Rob00},
Ciucu and Krattenthaler~\cite[Eqs.~(4.3) \& (4.5) with $x=0$]{CiuKra00},
Krattenthaler~\cite[Thm.~36 with $x=1$, $\mu=\nu=0$]{Kra99a},~\cite[Lem.~2 with $x=1$, $\mu=\nu=0$]{Kra99b}, 
and Kuperberg~\cite[Thm.~4, 9th \& 10th Eqs.~with $x=1$]{Kup02}.
Similarly,~\eqref{doubstair3} and~\eqref{doubstair4} are closely related to results of
Mills, Robbins and Rumsey~\cite[Thm.~5 with $x=\mu=1$]{MilRobRum87},~\cite[Thm.~2.1 with $x=\mu=1$]{Rob00},
Ciucu and Krattenthaler~\cite[Eqs.~(4.3) \& (4.5) with $x=1$]{CiuKra00},
Krattenthaler~\cite[Thm.~36 with $x=\mu=1$, $\nu=0$]{Kra99a},~\cite[Lem.~2 with $x=\mu=1$, $\nu=0$]{Kra99b},
and Kuperberg~\cite[Thm.~4, 1st \& 2nd Eqs.~with $x=1$]{Kup02}.

For a recent discussion of certain factorization identities for numbers of plane partitions, including~\eqref{PP1} and~\eqref{doubstair1},
see Ciucu~\cite[Sec.~1]{Ciu16}. For recent generalizations of the rhombus tiling interpretations of~\eqref{doubstair1}--\eqref{doubstair4}, 
see Ciucu~\cite[Thm.~8 \& Cor.~1]{Ciu18}, and Lai and Rohatgi~\cite[Thms.~2.1 \&~2.2]{LaiRoh18}.

\section*{Acknowledgements}
We thank Ilse Fischer, Ole Warnaar and anonymous referees for very helpful comments.
A.~Ayyer was partially supported by the UGC Centre for Advanced Studies and by Department of Science and Technology grants
DST/INT/SWD/VR/P-01/2014 and EMR/2016/006624. R.~E.~Behrend was partially supported by the Austrian Science Foundation FWF, START grant Y463.

\let\oldurl\url
\makeatletter
\renewcommand*\url{%
        \begingroup
        \let\do\@makeother
        \dospecials
        \catcode`{1
        \catcode`}2
        \catcode`\ 10
        \url@aux
}
\newcommand*\url@aux[1]{%
        \setbox0\hbox{\oldurl{#1}}%
        \ifdim\wd0>\linewidth
                \strut
                \\
                \vbox{%
                        \hsize=\linewidth
                        \kern-\lineskip
                        \raggedright
                        \strut\oldurl{#1}%
                }%
        \else
                \hskip0pt plus\linewidth
                \penalty0
                \box0
        \fi
        \endgroup
}
\makeatother
\gdef\MRshorten#1 #2MRend{#1}
\gdef\MRfirsttwo#1#2{\if#1M
MR\else MR#1#2\fi}
\def\MRfix#1{\MRshorten\MRfirsttwo#1 MRend}
\renewcommand\MR[1]{\relax\ifhmode\unskip\spacefactor3000 \space\fi
  \MRhref{\MRfix{#1}}{{\tiny \MRfix{#1}}}}
\renewcommand{\MRhref}[2]{
 \href{http://www.ams.org/mathscinet-getitem?mr=#1}{#2}}

\bibliography{Bibliography}
\bibliographystyle{amsplainhyper}
\end{document}